\documentclass[a4paper,11pt,final]{article}
\usepackage{fullpage}
\usepackage{graphicx}
\usepackage{amsthm, amssymb}
\usepackage{dsfont}
\usepackage{mathtools}
\usepackage{doi}
\usepackage{faktor}
\usepackage[font=small,labelfont=bf]{caption}
\usepackage[mathcal]{euscript}
\usepackage{mathrsfs}
\usepackage{stmaryrd}
\usepackage{verbatim}

\numberwithin{equation}{section}

\makeatletter							
\newtheorem*{rep@theorem}{\rep@title}
\newcommand{\newreptheorem}[2]{%
\newenvironment{rep#1}[1]{%
 \def\rep@title{#2 \ref{##1}}%
 \begin{rep@theorem}}%
 {\end{rep@theorem}}}
\makeatother

\let\amsamp=&

\begingroup
\catcode`\&=13
\gdef\pampmatrix{%
  \begingroup
  \let&=\amsamp
  \begin{pmatrix}%
}
\gdef\endpampmatrix{\end{pmatrix}\endgroup}
\endgroup

\newcommand{\kdef }{\ensuremath{\operatorname{k^{\textnormal{def}}}}}
\newcommand{\com }{\ensuremath{\operatorname{com}}}

\newcommand{\myvee }{\ensuremath{\bigvee}}

\newcommand{\op}{\textnormal{op}}

\newcommand{\Hom}{\textnormal{Hom}}

\newcommand{\tel}{\textnormal{tel}}

\DeclareMathOperator*\circprod{\bigcirc}

\theoremstyle{plain}
\newtheorem{theorem}{Theorem}[section]

\newreptheorem{theorem}{Theorem}
\newtheorem{proposition}[theorem]{Proposition}
\newreptheorem{proposition}{Proposition}
\newtheorem{lemma}[theorem]{Lemma}
\newreptheorem{lemma}{Lemma}
\newtheorem{corollary}[theorem]{Corollary}
\newreptheorem{corollary}{Corollary}

\theoremstyle{definition}
\newtheorem{definition}[theorem]{Definition}
\newtheorem{example}[theorem]{Example}

\theoremstyle{remark}
\newtheorem{remark}[theorem]{Remark}

\usepackage[all]{xy}
\SelectTips{cm}{10}
\CompileMatrices

\newcommand{\dmo}{\DeclareMathOperator}
\dmo{\colim}{colim}

\title{The spectrum for commutative complex $K$-theory}
\author{Simon Philipp Gritschacher}

\begin{document}

\maketitle

\begin{abstract}
We study commutative complex $K$-theory, a generalised cohomology theory built from spaces of ordered commuting tuples in the unitary groups. We show that the spectrum for commutative complex $K$-theory is stably equivalent to the $ku$-group ring of $BU(1)$ and thus obtain a splitting of its representing space $B_{\com}U$ as a product of all the terms in the Whitehead tower for $BU$, $B_{\com}U\simeq BU\times BU\langle 4\rangle \times BU\langle 6\rangle \times \dots .$ As a consequence of the spectrum level identification we obtain the ring of coefficients for this theory. Using the rational Hopf ring for $B_{\com}U$ we describe the relationship of our results with a previous computation of the rational cohomology algebra of $B_{\com}U$. This gives an essentially complete description of the space $B_{\com}U$ introduced by A. Adem and J. G\'omez.
\end{abstract}

\tableofcontents

\section{Introduction and results} \label{sec:introduction}

In this article we describe the unitary variant of commutative $K$-theory, which was introduced by Adem and G\'omez \cite{ademcommutativity}. The definition is based on a variation of the classical infinite loop space $BU$ which represents complex topological $K$-theory. Recall that a simplicial model for the classifying space $BG$ of any topological group $G$ is the geometric realisation of the nerve $N_\ast G$, where $G$ is regarded as a category with one object. Adem, Cohen and Torres-Giese \cite{ademcommutingelements} consider a natural subcomplex $B_{\com}G\subset BG$ encoding commutativity in $G$. It is obtained from $BG$ by restricting to those simplices which are ordered tuples of pairwise \emph{commuting} elements in $G$. More precisely, let
\[
C_k(G):=\{(g_1,\dots,g_k)\in G^k\,|\, g_ig_j=g_jg_i \textnormal{ for all }1\leq i,j\leq k\}\, ,
\]
considered as a subspace of the product $G^k$. Then it is easy to see that $C_\ast(G)\subset N_\ast G$ is a sub-simplicial space. It is common, in fact, to identify $C_k(G)$ with the set of group homomorphisms $\Hom(\mathbb{Z}^k,G)$, where $\mathbb{Z}^k$ is the free abelian group of rank $k$. Evaluating homomorphisms on the standard basis for $\mathbb{Z}^k$ gives a natural identification $\Hom(\mathbb{Z}^k,G)\cong C_k(G)$ and this induces a topology on $\Hom(\mathbb{Z}^k,G)$.

\begin{definition} \label{def:bcomg}
The \emph{classifying space for commutativity in $G$} is the geometric realisation $B_{\com}G:=|k\mapsto \Hom(\mathbb{Z}^k,G)|$. This construction is natural for homomorphisms of groups and there is a natural inclusion map $i: B_{\com}G\rightarrow BG$.
\end{definition}

This paper is concerned with the case $G=U$, where $U=\colim_{n}U(n)$ is the direct limit of the Lie groups $U(n)$ of unitary $n\times n$ matrices. Recall that $BU$ is an $E_{\infty}$-ring space without unit, whose addition and multiplication are induced from the direct sum and the tensor product in the unitary groups, respectively. Adem, G\'omez, Lind and Tillmann showed in \cite[Thm. 4.1]{ademnilpotentktheory} that the same operations induce an $E_{\infty}$-ring space structure on $B_{\com}U$ (which is called $B(2,U)$ in the following theorem).

\begin{theorem}[\cite{ademnilpotentktheory}]
The spaces $B(q,U)$, $B(q,SU)$, $B(q,SO)$, $B(q,O)$ and $B(q,Sp)$ provide a filtration by non-unital $E_{\infty}$-ring spaces of the classical non-unital $E_{\infty}$-ring spaces $BU$, $BSU$, $BSO$, $BO$ and $BSp$, respectively.
\end{theorem}

Thus they make the following definition\footnote{See Remark \ref{rem:representingspace}.} (cf. \cite{ademnilpotentktheory, ademcommutativity}).

\begin{definition} \label{def:ckth}
\emph{Commutative complex $K$-theory} $\tilde{K}_{\com}^\ast$ is the cohomology theory represented by the $E_{\infty}$-ring space $B_{\com}U$.
\end{definition}

An intriguing aspect of the construction $B_{\com}G$ is the fact that it can be used to parametrize an additional structure on principal $G$-bundles. In \cite{ademcommutativity} Adem and G\'omez introduce a notion of \emph{transitional commutativity}. Let $X$ be a CW complex. We say that a principal $G$-bundle $q: P\rightarrow X$ is transitionally commutative if there exists an open trivialising cover $\{U_i\,|\, i\in I\}$ of $X$ and a representing cocycle $\{g_{ij}: U_i\cap U_j\rightarrow G\,|\, i,j\in I\}$ for $q: P\rightarrow X$ so that for each $x\in X$ the set $\{g_{ij}(x)\,|\, i,j\in I: g_{ij}\textnormal{ is defined at }x\}\subset G$ is a subset of commuting elements. The authors then show that if $G$ is a Lie group and $f: X\rightarrow BG$ is the classifying map of a principal $G$-bundle $P\rightarrow X$, then $P$ is transitionally commutative if and only if $f$ can be factored, up to homotopy, through $i: B_{\com}G\rightarrow BG$. In particular, this gives $\tilde{K}^0_{\com}(X)=[X,B_{\com}U]$ an interpretation as a group of `transitionally commutative structures' on stable isomorphism classes of complex vector bundles over $X$ (cf. \cite[Thm. 5.5]{ademnilpotentktheory}).

The first objective of this paper is to determine the homotopy type of $B_{\com}U$. We will use the deformation $K$-theory (cf. \cite{carlssonstructuredstablehomotopy, lawsonproductformula}) of free abelian groups to define a spectrum for commutative $K$-theory. We then apply Lawson's Bott cofibre sequence in deformation $K$-theory \cite{lawsonbottcofibre} to determine the homotopy type of this spectrum. Thus our first result is:

\begin{reptheorem}{thm:spectrumbcomu}
There is a commutative $ku$-algebra spectrum $E$ which satisfies
\[
\Omega^\infty E\simeq \mathbb{Z}\times B_{\com}U\sslash U
\]
and an equivalence of commutative $ku$-algebras $E\simeq ku\wedge BU(1)_+$.
\end{reptheorem}

Here $ku$ denotes the connective $K$-theory spectrum. The space $B_{\com}U$ has an action of $U$ simplicialwise by conjugation, and $B_{\com}U\sslash U$ is the resulting homotopy orbit space. A line bundle has a natural `transitionally commutative structure' as $BU(1)=B_{\com}U(1)$. The equivalence in the theorem is induced by embedding $BU(1)$ in the natural way in the space $B_{\com}U$ and then into $\{1\}\times B_{\com}U\sslash U$, corresponding to the inclusion $\{\textnormal{line bundles}\}\subset \{\textnormal{transitionally commutative bundles}\}$. Thus one can say that commutative $K$-theory is the free $ku$-algebra theory `generated by line bundles'.

The theorem shows that $\pi_\ast(E)$ is the connective $K$-Pontrjagin ring of $BU(1)$. The structure of this ring is well known. Let $x\in ku^{2}(BU(1))$ be the standard choice of complex orientation for $ku$-theory, and let $y_n \in ku_{2n}(BU(1))$ be dual to $x^n\in ku^{2n}(BU(1))$. Then $ku_\ast(BU(1))$ is a free $\pi_\ast(ku)$-module on generators $1$ and $y_n$ for $n\geq 1$. In the algebra $ku_\ast(BU(1))$ the $y_n$ satisfy certain relations (cf. Section \ref{sec:homotopy}).

\begin{repcorollary}{cor:homotopygroups}
The homotopy ring $\pi_\ast(B_{\com}U)$ is the ideal $(y_n\,|\, n\geq 1)\subset ku_\ast(BU(1))$. In particular, the homotopy groups of the space $B_{\com}U$ are as follows,
\begin{alignat*}{2}
\pi_{2n}(B_{\com}U)		&= \mathbb{Z}^n, \\
\pi_{2n+1}(B_{\com}U)	&= 0
\end{alignat*}
for all $n\geq 0$.
\end{repcorollary}

We also describe the homotopy ring of $B_{\com}SU$ (Corollary \ref{cor:homotopybcomsu}).

The inclusion map $i: B_{\com}U\rightarrow BU$ induces a transformation of multiplicative cohomology theories from commutative to ordinary $K$-theory. In \cite[Thm. 4.2]{ademnilpotentktheory} it is shown that this map has a section by an infinite loop map, so that $BU$ is a direct factor of $B_{\com}U$. For any $G$ let $E_{\com}G$ denote the homotopy fibre of the map $i: B_{\com}G\rightarrow BG$.

\begin{theorem}[{\cite{ademnilpotentktheory}, case $q=2$}]
For $G=U$, $SU$, $SO$, $O$ and $Sp$ there is a homotopy split fibration of infinite loop spaces $E_{\com}G\rightarrow B_{\com}G\rightarrow BG$. In particular, there is a splitting of spaces $B_{\com}G\simeq BG\times E_{\com}G$.
\end{theorem}

Using the fact that $ku\wedge BU(1)_+$ splits as a wedge of suspensions of $ku$, we obtain the following strengthening of their theorem for $G=U$ and $SU$: Let $BU\langle 2n\rangle\rightarrow BU$ denote the $2n-1$ connected cover of $BU$.

\begin{reptheorem}{thm:splitting}
There are splittings of $E_{\infty}-\mathbb{Z}\times BU$-modules $B_{\com}U \simeq BU \times E_{\com}U$, and $B_{\com}SU \simeq BSU \times E_{\com}U$, and
\[
E_{\com}U \simeq \prod_{n\geq 2} BU\langle 2n\rangle\, .
\]
\end{reptheorem}

Prior to this work, Adem and G\'omez had computed the rational cohomology algebra of $B_{\com}U$ by regarding it as the direct limit of all $B_{\com}U(n)$ for $n\geq 1$. It is a polynomial algebra $\mathbb{Q}[z_{a,b}\,|\,(a,b)\in \mathbb{N}^2,\, b\geq 1]$, where the $z_{a,b}$ are generators of degree $2(a+b)$. Their computation relies on Lie group theory and shows that the classes $z_{a,b}$ are related to \emph{multisymmetric functions} in the same way as the components of the Chern character are related to elementary symmetric polynomials. It seems interesting now to compare the description of the homotopy groups of $B_{\com}U$ via $K$-homology theory to the description of the cohomology groups via multisymmetric functions. To do this we describe the rational \emph{Hopf ring} of $B_{\com}U$ in a basis determined by the $z_{a,b}$.

Let $\zeta_{a,b}$ be dual to $z_{a,b}$ and let $[n]\in H_{0}(\mathbb{Z},\mathbb{Q})$ be the homology class determined by $n\in \mathbb{Z}$. We write $\circ$ for the Hopf ring multiplication.
\begin{reptheorem}{thm:hopfring}
The Hopf ring $H_{\ast}(\mathbb{Z}\times B_{\com}U\sslash U,\mathbb{Q})$ is generated by $[1]\otimes 1$, and the classes $[0]\otimes \zeta_{1,0}$ and $[0]\otimes \zeta_{0,1}$ in degree two. In this ring the class $\zeta_{a,b}:=[0]\otimes \zeta_{a,b}$ has the presentation
\[
\zeta_{a,b}=\frac{\zeta_{1,0}^{\circ a}\circ \zeta_{0,1}^{\circ b}}{a!b!}\,,\quad (a,b)\in \mathbb{N}^2-\{(0,0)\}\, .
\]
\end{reptheorem}

Using this theorem we obtain some interesting formulas. For example:
\begin{repcorollary}{cor:hurewicz}
The Hurewicz homomorphism $h: \pi_\ast(B_{\com}U)\rightarrow \tilde{H}_\ast(B_{\com}U,\mathbb{Q})$ is determined by the formula
\[
h(y_n)=\sum_{j=0}^{n-1} \frac{s(n,n-j)}{\binom{n}{j}} \,\zeta_{j,n-j}\,, \quad n\geq 1\,,
\]
where the $s(n,n-j)$ are the Stirling numbers of the first kind.
\end{repcorollary}

We also describe the splitting equivalence in Theorem \ref{thm:splitting} on rational cohomology: For $k\geq 1$ let $f_k: B_{\com}U \rightarrow BU$ be the composition of the splitting, the projection onto $BU\langle 2k\rangle$ and the covering $BU\langle 2k\rangle\rightarrow BU$. Let $\textnormal{ch}_n\in H^{2n}(B_{\com}U,\mathbb{Q})$ be the $n$-th component of the Chern character.

\begin{repcorollary}{cor:splittingoncohomology}
For all $k\geq 1$ and all $n\geq 1$ the following formula holds
\[
n!\, f_k^\ast(\textnormal{ch}_n)=k! \,\sum_{j=0}^{n-1} \binom{n}{j}\,S(n-j,k)\, z_{j,n-j}\, ,
\]
where the $S(n-j,k)$ are the Stirling numbers of the second kind.
\end{repcorollary}

Finally, we discuss the relationship between the integral cohomology of $B_{\com}U$ and that of $B_{\com}SU(2)$ relying on joint work in progress \cite{antolinbcomu2}. The first interesting commutative $K$-group of spheres is $\tilde{K}^0_{\com}(S^4)\cong \mathbb{Z}\oplus \mathbb{Z}$. In the last part of the paper we use cohomology to prove:
\begin{repproposition}{prop:kcoms4}
The natural map $B_{\com}SU(2)\rightarrow B_{\com}U$ induces an isomorphism
\[
\pi_4(B_{\com}SU(2))\cong \tilde{K}^0_{\com}(S^4)\,,
\]
and the map $B_{\com}U(2)\rightarrow B_{\com}U$ is $4$-connected.
\end{repproposition}

\paragraph{Outline.} In Section \ref{sec:spectrum} we use deformation $K$-theory to represent commutative $K$-theory by a commutative $ku$-algebra spectrum, prove Theorem \ref{thm:spectrumbcomu} and deduce the homotopy ring of $B_{\com}U$. In Section \ref{sec:type} we identify $B_{\com}SU$ as the homotopy fibre of the determinant map, prove Theorem \ref{thm:splitting} and thus identify the homotopy type of $B_{\com}U$ as well as $B_{\com}SU$. We include furthermore a brief discussion of cohomology operations in commutative $K$-theory. In Section \ref{sec:cohomologybcomu} we turn to cohomological results. We prove Theorem \ref{thm:hopfring} which describes the rational Hopf ring of $B_{\com}U$. We use this to describe the rational Hurewicz map, the splitting of $B_{\com}U$ on rational cohomology, and the canonical map $B_{\com}U(2)\rightarrow B_{\com}U$ on integral cohomology. In Section \ref{sec:kcoms4} we prove Proposition \ref{prop:kcoms4} by constructing classifying maps for transitionally commutative $SU(2)$-bundles and computing their characteristic classes. Finally, in Appendix \ref{sec:modulestructure} we offer a description of the inclusion map $B_{\com}G_{\mathds{1}}\rightarrow BG$ on rational cohomology.

\paragraph{Acknowledgements.} This paper contains material of my doctoral thesis which was written at the University of Oxford. I would like to thank my supervisor Ulrike Tillmann for introducing me to the topic, and for all her encouragement and help. I would also like to thank Andr\'e Henriques for a clarifying discussion and Graeme Segal for his useful comments. This work also profited from conversations with Bernardo Villarreal and Omar Antol\'in Camarena.

\section{The spectrum for commutative $K$-theory} \label{sec:spectrum}

In order to analyse the space $B_{\com}U$ we use a model based on the deformation $K$-theory of free abelian groups.

\subsection{Deformation $K$-theory} \label{sec:deformationktheory}

Suppose that $\pi$ is a finitely generated discrete group. We may consider the category $\mathscr{R}(\pi)$ of finite dimensional representations of $\pi$ and their isomorphisms. It is naturally a topological category, that is, it has an object and a morphism space, so that domain, codomain, identity and composition maps are all continuous (cf. \cite{lawsonproductformula}). Namely, if $\mathcal{S}\subset \pi$ is a generating set for $\pi$ and $G$ is a topological group, then $\Hom(\pi,G)$ can always be topologised as a subspace of the finite product $G^\mathcal{S}$. In addition, the category $\mathscr{R}(\pi)$ can be given the structure of a permutative category with monoidal product induced by the direct sum of representations.

\begin{definition}
The \emph{deformation $K$-theory} of $\pi$ is the $K$-theory of the permutative topological category $\mathscr{R}(\pi)$.
\end{definition}
The study of this $K$-theory was first suggested by Carlsson  \cite{carlssonstructuredstablehomotopy}. For example, considering representations in the unitary groups, the category $\mathscr{R}(\pi)$ is the `action groupoid' for the action of $U(n)$ on $\Hom(\pi,U(n))\subset U(n)^{\mathcal{S}}$ by conjugation,
\[
\mathscr{R}(\pi)=\coprod_{n\geq 0} U(n)\ltimes \Hom(\pi,U(n))\,.
\]
Let us write $\kdef(\pi)$ for the unitary deformation $K$-theory spectrum of $\pi$. For example, if $\pi=1$ is the trivial group, then $\kdef(1)$ is the $K$-theory spectrum of the category of finite dimensional unitary vector spaces and isometries, so it is a model for $ku$-theory.

In \cite{lawsonproductformula} Lawson explains how the tensor product of representations gives $\kdef(\pi)$ the structure of a commutative ring spectrum. In fact, he shows that $\kdef(\pi)$ is a $ku$-algebra spectrum via the unit map $ku\simeq \kdef (1)\rightarrow \kdef (\pi)$ induced by the homomorphism $\pi\rightarrow 1$. More precisely, he constructs a functor
\begin{equation} \label{eq:deformationktheory}
\kdef: \left\{\begin{matrix} \textnormal{finitely generated} \\ \textnormal{discrete groups} \end{matrix}\right\}^{\op} \longrightarrow  \textnormal{com-Alg}_{\,ku}\,,
\end{equation}
from finitely generated discrete groups and homomorphisms into the category of commutative $ku$-algebras in symmetric spectra (cf. \cite{hoveysymmetricspectra}).

Let us explain this construction in a bit more detail. Let $\mathbb{C}^\infty$ be equipped with the standard inner product. Let $\Gamma^{\op}$ be the category of finite sets with basepoint and based functions between them. Associated to the group $\pi$ is a $\Gamma$-space (cf. \cite{segalcohomologytheories}) $\mathcal{K}(\pi): \Gamma^{\op}\rightarrow \mathbf{Top}_\ast$ which takes a finite pointed set $S\in \Gamma^{\op}$ to the space
\[
\mathcal{K}(\pi)(S):=\left\{  (W_a,\rho_a)_{a\in S}\ \middle\vert \begin{array}{l}
   W_a\subset \mathbb{C}^\infty \textnormal{ a finite dimensional inner product space}\\
   \rho_a: \pi\rightarrow U(W_a) \textnormal{ a unitary representation} \\
   W_a\perp W_b \textnormal{ if } a\neq b, \ W_a=\{0\}\textnormal{ if }a\textnormal{ is the basepoint}
  \end{array}\right\}\, .
\]
The space $\mathcal{K}(\pi)(S)$ is topologised as a subspace of the product
\[
\left( \coprod_{n\geq 0} V_n\times_{U(n)}\Hom(\pi,U(n))\right)^S\,,
\]
where $V_n$ is the `Stiefel manifold' of orthonormal $n$-frames in $\mathbb{C}^\infty$ and the action of $U(n)$ on $\Hom(\pi,U(n))$ is by isomorphisms of representations (or by conjugation, if we identify $\Hom(\pi,U(n))$ with a subspace of a finite product of copies of $U(n)$). Given a morphism $\alpha: S\rightarrow T$ in $\Gamma^\op$ a map $\mathcal{K}(\pi)(\alpha): \mathcal{K}(\pi)(S)\rightarrow \mathcal{K}(\pi)(T)$ can be defined by
\[
\mathcal{K}(\pi)(\alpha)((W_a,\rho_a)_{a\in S}):=\left( \bigoplus_{a\in \alpha^{-1}(b)} W_a, \bigoplus_{a\in \alpha^{-1}(b)}\rho_a \right)_{b\in T}\,,
\]
which makes sense because the $W_a$ are mutually orthogonal inside $\mathbb{C}^\infty$. The $\Gamma$-space $\mathcal{K}(\pi)$ is special. In fact, it is the $\Gamma$-space associated - in the manner described in \cite[\S 2]{segalcohomologytheories} - to the permutative category $\mathscr{R}(\pi)$. Thus the symmetric spectrum associated to $\mathcal{K}(\pi)$ is an $\Omega$-spectrum above the zero space. We denote this spectrum by $\mathcal{K}(\pi)(\mathbb{S})$ (which means $\mathcal{K}(\pi)$ evaluated on the sphere spectrum).

For $k\geq 0$ let $\mathscr{L}((\mathbb{C}^\infty)^{\otimes k},\mathbb{C}^\infty)$ denote the space of linear isometric embeddings $(\mathbb{C}^{\infty})^{\otimes k}\hookrightarrow \mathbb{C}^\infty$ with the compact-open topology. Lawson observed that the tensor product of representations yields continuous and natural multiplication maps
\[
\mathscr{L}((\mathbb{C}^\infty)^{\otimes k},\mathbb{C}^\infty)_+ \wedge \underset{\stackrel{\longleftrightarrow}{k}}{\mathcal{K}(\pi)\wedge \dots \wedge \mathcal{K}(\pi)}\stackrel{\otimes}{\longrightarrow} \mathcal{K}(\pi)\,,
\]
where the smash product is formed in the topologically enriched permutative category of $\Gamma$-spaces (cf. \cite{lydakissmashproduct}). These multiplication maps can be fed into the machine of Elmendorf and Mandell \cite{elmendorfinfiniteloopspace} which associates to $\mathcal{K}(\pi)$ a so-called $E_{\infty}$-$ku$-algebra object in the category of symmetric spectra (to be precise, here symmetric spectra means symmetric spectra of \emph{simplicial sets}). The general theory of \cite{elmendorfinfiniteloopspace} allows one to rigidify this spectrum, which then yields the functor $\kdef(-)$ in (\ref{eq:deformationktheory}). This procedure is described in detail in \cite[\S 7]{lawsonproductformula}. In particular, the infinite loop space associated to $\kdef(\pi)$ is weakly equivalent to $\Omega^\infty \mathcal{K}(\pi)(\mathbb{S})$.

One may also consider the category of isomorphisms classes of representations
\[
\faktor{\mathscr{R}(\pi)}{\cong} =\coprod_{n\geq 0} \Hom(\pi,U(n))/U(n)\,,
\]
with only identity morphisms. This is a strictly commutative topological ring under direct sum and tensor product of representations. It gives rise to a $\Gamma$-space $\mathcal{R}(\pi)$ and there is an obvious map of $\Gamma$-spaces $\mathcal{K}(\pi)\rightarrow \mathcal{R}(\pi)$ sending a representation to its isomorphism class.

\begin{definition}
The symmetric spectrum determined by $\mathcal{R}(\pi)$ is called the \emph{deformation representation ring spectrum} of $\pi$ (cf. \cite{lawsoncompleted}).
\end{definition}
The deformation representation ring can be modelled as a commutative $H\mathbb{Z}$-algebra in symmetric spectra and is denoted by $R[\pi]$. Via the canonical map $ku\rightarrow H\mathbb{Z}$ the spectrum $R[\pi]$ can also be regarded as a commutative $ku$-algebra. There is then a natural transformation of functors $\kdef(-)\rightarrow R[-]$.

In \cite{lawsonbottcofibre} Lawson shows that $R[\pi]$ is the cofibre of `multiplication by the Bott element' on $\kdef(\pi)$. Let $u\in \pi_2(ku)$ be the Bott class. Using the structure of $\kdef(\pi)$ as a $ku$-module, there is a map $u\cdot - $ defined as the composite
\[
S^2\wedge \kdef(\pi)\xrightarrow{u\wedge id} ku\wedge \kdef(\pi) \longrightarrow \kdef(\pi)\, .
\]

\begin{theorem}[\cite{lawsonbottcofibre}] \label{thm:bottcofibre}
There is a homotopy cofibre sequence of $ku$-modules
\[
\Sigma^2 \kdef (\pi)\stackrel{u\cdot }{\longrightarrow} \kdef (\pi)\longrightarrow R[\pi]\,,
\]
where the first map is `multiplication by the Bott element'.
\end{theorem}

\subsection{The spectrum $\kdef(\mathbb{Z}^\ast)$}

We now define a spectrum for $B_{\com}U$. For this it is convenient to co-represent the simplicial space $k\mapsto \Hom(\mathbb{Z}^k,U)$. Let $F_k$ denote the free group on $k$ generators $x_1,\dots,x_k$.

\begin{definition}
\label{def:cosimplicialobject}
Define a cosimplicial group $F_\ast: \Delta\rightarrow \mathbf{Grp}$ with co-face maps $d^i: F_{k-1}\rightarrow F_{k}$ given by
\begin{equation*}
d^ix_j=\begin{cases} x_j\,, & j<i \\ x_jx_{j+1}\,,& j=i \\ x_{j+1}\,,& j>i\end{cases}
\end{equation*}
for $0\leq i\leq k$ and $k>0$ and co-degeneracy maps $s^i: F_{k+1}\rightarrow F_k$ given by
\begin{equation*}
s^ix_j=\begin{cases} x_j\,, & j<i+1 \\ 1\,,& j=i+1 \\ x_{j-1}\,,& j>i+1\end{cases}
\end{equation*}
for $0\leq i\leq k$ and $k\geq 0$. Another cosimplicial group $\mathbb{Z}^\ast$ sending $k\mapsto \mathbb{Z}^k$ is defined by simplicialwise abelianisation of $F_{\ast}$. It comes with a morphism $F_{\ast}\rightarrow \mathbb{Z}^{\ast}$.
\end{definition}

It is easily verified that the composite $\Hom(F_\ast,U):=\Hom(-,U)\circ F_\ast$ is the simplicial bar construction for $U$, and $\Hom(\mathbb{Z}^\ast,U)$ is the simplicial space whose realisation is $B_{\com}U$. The morphism of cosimplicial groups $F_\ast\rightarrow \mathbb{Z}^\ast$ induces the canonical map $i: B_{\com}U\rightarrow BU$.

\begin{definition} \label{def:spectrume}
Define a commutative $ku$-algebra $E:=|\kdef(\mathbb{Z}^\ast)|$.
\end{definition}

Recall that the deformation $K$-theory spectrum $\kdef(\mathbb{Z}^k)$ comes from a $\Gamma$-space $\mathcal{K}(\mathbb{Z}^k)$. The assignment $k\mapsto \mathcal{K}(\mathbb{Z}^k)$ defines a simplicial $\Gamma$-space. Let $|\mathcal{K}(\mathbb{Z}^\ast)|$ be the $\Gamma$-space whose value on a finite set $S\in \Gamma^{\op}$ is the geometric realisation $|\mathcal{K}(\mathbb{Z}^\ast)|(S):=|k\mapsto \mathcal{K}(\mathbb{Z}^k)(S)|$.

\begin{lemma} \label{lem:loopsinfinity1}
There is a zig-zag of stable equivalences between $E$ and the symmetric spectrum associated to the $\Gamma$-space $|\mathcal{K}(\mathbb{Z}^\ast)|$.
\end{lemma}
\begin{proof}
Let $\textnormal{Sing}$ denote the singular complex functor from topological spaces to simplicial sets. The rigidification in \cite[\S 7]{lawsonproductformula} produces a zig-zag of stable equivalences between $\kdef(\mathbb{Z}^k)$ and the symmetric spectrum associated to the $\Gamma$-space $\textnormal{Sing}\circ \mathcal{K}(\mathbb{Z}^k)$. Furthermore, the maps in the zig-zag are natural with respect to homomorphisms of groups, so there is a zig-zag of maps of simplicial symmetric spectra
\[
(\textnormal{Sing}\circ \mathcal{K}(\mathbb{Z}^\ast))(\mathbb{S}) \leftarrow \cdots \rightarrow \kdef(\mathbb{Z}^\ast)
\]
which are stable equivalences in each simplicial degree. By \cite[Cor. 4.1.6]{shipleysymmetricspectra} this induces a stable equivalence on geometric realisations, i.e. between $E$ and $|\textnormal{Sing}\circ \mathcal{K}(\mathbb{Z}^\ast)|(\mathbb{S})$. By adjunction, we have a map $|\textnormal{Sing}\circ \mathcal{K}(\mathbb{Z}^\ast)|(\mathbb{S})\rightarrow |\mathcal{K}(\mathbb{Z}^\ast)|(\mathbb{S})$ which we need to show is a stable equivalence. Since both spectra are $\Omega$-spectra above the zero space, it suffices to check that the map on level one spaces $|\textnormal{Sing}\circ \mathcal{K}(\mathbb{Z}^\ast)|(S^1)\rightarrow |\mathcal{K}(\mathbb{Z}^\ast)|(S^1)$ is a weak homotopy equivalence. The map is a weak equivalence in every simplicial degree, so the result follows from \cite[Thm. A.4]{mayeinftyspaces} and the fact that $k\mapsto \mathcal{K}(\mathbb{Z}^k)(S^1)$ is a proper simplicial space (cf. Lemma \ref{lem:proper}).
\end{proof}

\begin{lemma} \label{lem:proper}
The simplicial space $k\mapsto \mathcal{K}(\mathbb{Z}^k)(S^1)$ is proper.
\end{lemma}
\begin{proof}
Let $\mathbf{m}^+\in \Gamma^{\op}$ be the set $\{0,1,\dots,m\}$ with $0$ as basepoint. We first show that for fixed $m\geq 0$ the simplicial space $k\mapsto \mathcal{K}(\mathbb{Z}^k)(\mathbf{m}^+)$ is proper. Let $V_n$ be the Stiefel manifold of orthonormal $n$-frames in $\mathbb{C}^\infty$. By definition, $\mathcal{K}(\mathbb{Z}^k)(\mathbf{m}^+)$ is a disjoint union over spaces of the form
\[
V_n\times_{U(n_1)\times \cdots \times U(n_m)}\Hom(\mathbb{Z}^k,U(n_1))\times \cdots \times \Hom(\mathbb{Z}^k,U(n_m))\,,
\]
where the $n_i$ are non-negative integers with $\sum_{i=1}^m n_i=n$. Let $\underline{n}=(n_1,\dots,n_m)$ and write $U(\underline{n})=U(n_1)\times \cdots \times U(n_m)$. Let $S_k\subset \Hom(\mathbb{Z}^k,U(\underline{n}))$ be the space of degenerate $k$-simplices in the simplicial space $\Hom(\mathbb{Z}^\ast,U(\underline{n}))$. It is proved in \cite[Thm. 4.8]{ademstablesplittings} (in their notation, we choose $G=U(\underline{n})$, $K=1$ and $r=1$) that the pair of spaces $(\Hom(\mathbb{Z}^k,U(\underline{n})),S_k)$ is a strong $U(\underline{n})$-equivariant NDR pair. It follows that the pair $(V_n\times_{U(\underline{n})} \Hom(\mathbb{Z}^k,U(\underline{n})),V_n\times_{U(\underline{n})}S_k)$ is strongly NDR and, therefore, $k\mapsto \mathcal{K}(\mathbb{Z}^k)(\mathbf{m}^+)$ is proper.

Let $\Delta^1/\partial\Delta^1$ be the simplicial circle with $m+1$ simplices in degree $m$. For every fixed $k\geq 0$, $\mathcal{K}(\mathbb{Z}^k)(\Delta^1/\partial \Delta^1)$ is a simplicial space, in which the degeneracy maps $\mathcal{K}(\mathbb{Z}^k)(\mathbf{m}^+)\rightarrow \mathcal{K}(\mathbb{Z}^k)(\mathbf{m+1}^+)$ are inclusions of connected components, thus closed cofibrations. It follows that the bisimplicial space $\mathcal{K}(\mathbb{Z}^\ast)(\Delta^1/\partial\Delta^1)$ is proper in the simplicial direction of $\mathbb{Z}^\ast$ and good in the simplicial direction of the circle. Now we use the fact that levelwise cofibrations between good simplicial spaces induce a cofibration on realisations (see \emph{e.g.} \cite[\S 14-5]{warnerhomotopytheory}) to see that if we realise first in the `good' direction, the resulting simplicial space $k\mapsto \mathcal{K}(\mathbb{Z}^k)(S^1)$ is proper.
\end{proof}

\begin{lemma} \label{lem:loopsinfinity2}
There is a weak equivalence $\Omega^\infty |\mathcal{K}(\mathbb{Z}^\ast)|(\mathbb{S})\simeq \mathbb{Z}\times B_{\com}U\sslash U$.
\end{lemma}
\begin{proof}
The $\Gamma$-space $|\mathcal{K}(\mathbb{Z}^\ast)|$ is the one associated to the permutative category
\begin{equation} \label{eq:translationcat}
\coprod_{n\geq 0}U(n)\ltimes B_{\com}U(n)\,,
\end{equation}
which is the action groupoid for the action of $U(n)$ on $B_{\com}U(n)$ by conjugation. The infinite loop space $\Omega^\infty |\mathcal{K}(\mathbb{Z}^\ast)|(\mathbb{S})=\Omega |\mathcal{K}(\mathbb{Z}^\ast)|(S^1)$ is then the group-completion of the classifying space of (\ref{eq:translationcat}). The classifying space of (\ref{eq:translationcat}) is the homotopy orbit $\coprod_{n\geq 0} B_{\com}U(n)\sslash U(n)$, so by the group-completion theorem \cite{mcduffsegalgroupcompletion} we get a map
\[
\mathbb{Z}\times \tel_{n\to \infty} \,B_{\com}U(n)\sslash U(n)\rightarrow \Omega^\infty  |\mathcal{K}(\mathbb{Z}^\ast)|(\mathbb{S})\,,
\]
which is an integer homology equivalence. In fact, it is a weak equivalence, because the telescope is simply connected: For all $n\geq 1$, the space $B_{\com}U(n)\sslash U(n)$ is simply connected, because it is the geometric realisation of a simplicial space with only one vertex and a connected space of $1$-simplices. Finally, $\tel_{n\to \infty} \,B_{\com}U(n)\sslash U(n)\simeq B_{\com}U\sslash U$. This follows by first commuting the telescope with geometric realisation, then replacing the telescope by a colimit (using the fact that the maps $\Hom(\mathbb{Z}^k,U(n))\rightarrow \Hom(\mathbb{Z}^k,U(n+1))$ are cofibrations), and finally identifying $\colim_n \Hom(\mathbb{Z}^k,U(n))\cong \Hom(\mathbb{Z}^k,U)$, since $\mathbb{Z}^k$ is finitely generated.
\end{proof}

We now identify the spectrum $E$.

\begin{lemma} \label{lem:rzk}
There is a stable equivalence of commutative $H\mathbb{Z}$-algebras
\[
|R[\mathbb{Z}^\ast]|\simeq H\mathbb{Z}\wedge BU(1)_+\,.
\]
\end{lemma}
\begin{proof}
Fix $k\geq 0$. The symmetric spectrum $R[\mathbb{Z}^k]$ is determined by the $\Gamma$-space associated to the abelian semi-ring $\coprod_{n\geq 0}\Hom(\mathbb{Z}^k,U(n))/U(n)$. The spectral theorem implies that there is a natural homeomorphism
\begin{equation} \label{eq:repzk}
\coprod_{n\geq 0}\Hom(\mathbb{Z}^k,U(n))/U(n)\cong SP^\infty(U(1)^k_+)\,,
\end{equation}
where $SP^\infty(U(1)^k_+)$ is the free abelian monoid generated by the space $U(1)^k$ with a disjoint basepoint as additive unit (see \emph{e.g.} \cite[Thm. 6.1]{ademstablesplittings} for a proof of (\ref{eq:repzk})). There is a multiplication map $U(1)^k_+\wedge U(1)^k_+\cong (U(1)^k\times U(1)^k)_+\rightarrow U(1)^k_+$ using the group structure of $U(1)^k$. This defines a multiplication map
\[
SP^\infty(U(1)^k_+)\wedge SP^\infty(U(1)^k_+)\rightarrow SP^\infty(U(1)^k_+)
\]
via $\left(\sum_i n_i x_i\right)\wedge \left(\sum_k m_k y_k\right)\mapsto \sum_{i,k} n_im_k\, x_i y_k$ making $SP^\infty(U(1)^k_+)$ into an abelian semi-ring in such a way that (\ref{eq:repzk}) is an isomorphism of semi-rings. By the Dold-Thom theorem $SP^\infty(U(1)^k_+)$ represents the reduced integral Pontrjagin ring of $U(1)^k_+$. The commutative ring spectrum associated to $SP^\infty(U(1)^k_+)$ is stably equivalent to $H\mathbb{Z}\wedge U(1)^k_+$ as an $H\mathbb{Z}$-algebra. The stable equivalence $|R[\mathbb{Z}^\ast]|\simeq H\mathbb{Z}\wedge BU(1)_+$ follows from this after geometric realisation.
\end{proof}

\begin{corollary} \label{cor:ecofibre}
There is a homotopy cofibre sequence of $ku$-modules
\[
\Sigma^2 E\stackrel{u\cdot }{\longrightarrow} E\stackrel{\lambda}{\longrightarrow} H\mathbb{Z}\wedge BU(1)_+\,,
\]
in which the first map is multiplication by the Bott element.
\end{corollary}
\begin{proof}
This follows from Theorem \ref{thm:bottcofibre} applied to $\mathbb{Z}^\ast$ and geometric realisation, in view of Lemma \ref{lem:rzk}.
\end{proof}

We can now prove the main result of this section:

\begin{theorem} \label{thm:spectrumbcomu}
There is a commutative $ku$-algebra spectrum $E$ which satisfies
\[
\Omega^\infty E\simeq \mathbb{Z}\times B_{\com}U\sslash U
\]
and a stable equivalence of commutative $ku$-algebras $E\simeq ku\wedge BU(1)_+$.
\end{theorem}
\begin{proof}
The first part follows directly from Lemmas \ref{lem:loopsinfinity1} and \ref{lem:loopsinfinity2}.

There is a map of commutative ring spectra $j: \Sigma^{\infty}BU(1)_+\rightarrow E$ which is induced by the canonical map $BU(1)\rightarrow B_{\com}U$. Recall from Section \ref{sec:deformationktheory} that if $S\in \Gamma^\op$, then a point in the space $\mathcal{K}(\mathbb{Z}^k)(S)$ is an $S$-indexed tuple $(W_a,\rho_a)_{a\in S}$ of finite dimensional mutually orthogonal inner product spaces $W_a\subset \mathbb{C}^\infty$ with representations $\rho_a: \mathbb{Z}^k\rightarrow U(W_a)$ on them. Now let $\Gamma_{U(1)^k}$ be the $\Gamma$-space which is obtained from $\mathcal{K}(\mathbb{Z}^k)$ by specifying in addition the data of an unordered orthonormal frame $\{w_{a,1},\dots,w_{a,n_a}\}\subset W_a$, where $n_a=\dim\,W_a$, for each $a\in S$ so that the representation $\rho_a$ is diagonal with respect to this frame. The $\Gamma$-space $\Gamma_{U(1)^k}$ has as underlying $H$-space
\[
\Gamma_{U(1)^k}(\mathbf{1}^+)=\coprod_{n\geq 0} V_n\times_{\Sigma_n} (U(1)^k)^n\,,
\]
the free $E_{\infty}$-algebra on the space $U(1)^k$. Moreover, it has multiplication maps induced from tensor product and thus leads to the group ring spectrum $\Gamma_{U(1)^k}(\mathbb{S})=\Sigma^{\infty}U(1)^k_+$. There is now an obvious forgetful map of $\Gamma$-spaces $\Gamma_{U(1)^k}\rightarrow \mathcal{K}(\mathbb{Z}^k)$ yielding a map of commutative ring spectra $\Sigma^\infty U(1)^k_+\rightarrow \kdef (\mathbb{Z}^k)$. The map $j: \Sigma^\infty BU(1)_+\rightarrow E$ is then induced on geometric realisations.

The composite map $\Sigma^{\infty}BU(1)_+\stackrel{j}{\longrightarrow} E\stackrel{\lambda}{\longrightarrow}H\mathbb{Z}\wedge BU(1)_+$ is easily seen to be the natural one, representing the stable Hurewicz map for $BU(1)$. Since $E$ is a $ku$-algebra, $j$ extends over the free $ku$-algebra $ku\wedge BU(1)_+$. Thus we obtain a sequence
\[
ku\wedge BU(1)_+\stackrel{j'}{\longrightarrow} E\stackrel{\lambda}{\longrightarrow} H\mathbb{Z}\wedge BU(1)_+\,,
\]
whose composite is the smash product of the linearization map $ku\rightarrow H\mathbb{Z}$ with $BU(1)_+$. If we combine this with Corollary \ref{cor:ecofibre}, we obtain a map of homotopy cofibre sequences
\[
\xymatrix{
\Sigma^2 ku\wedge BU(1)_+ \ar[r]\ar[d]^-{\Sigma^2j'} & ku\wedge BU(1)_+ \ar[r] \ar[d]^-{j'}	& H\mathbb{Z}\wedge BU(1)_+ \ar@{=}[d] \\
\Sigma^2E \ar[r] 	& E \ar[r]^-{\lambda} &	H\mathbb{Z}\wedge BU(1)_+
}
\]
where the top cofibering is obtained from the Bott periodicity sequence by smashing with $BU(1)_+$. Inductively, using the five lemma, we see that $j'$ induces an isomorphism of homotopy groups, hence is a stable equivalence.
\end{proof}

\begin{remark}\label{rem:representingspace}
It appears that either the space $B_{\com}U$ or $\mathbb{Z}\times B_{\com}U\sslash U$ is the natural output of a multiplicative infinite loop space machine, but not $\mathbb{Z}\times B_{\com}U$. We therefore prefer to think of commutative $K$-theory $\tilde{K}^\ast_{\com}$ as being represented by the space $B_{\com}U$, even though this does not quite agree with the definitions in \cite{ademnilpotentktheory} (where $\mathbb{Z}\times B_{\com}U$ is taken to represent the unreduced $K$-theory).
\end{remark}

\subsection{The homotopy groups of $B_{\com}U$} \label{sec:homotopy}

The structure of the $K$-homology ring $ku_\ast(BU(1))$ is well known. If we think of $BU$ as the second term in an $\Omega$-spectrum for $ku$, then the canonical map $BU(1)\rightarrow BU$ determines a class $x\in ku^2(BU(1))$. Let $y_n\in ku_{2n}(BU(1))$ be dual to $x^n\in ku^{2n}(BU(1))$. Let $\mathbb{Z}[u]\cong \pi_\ast(ku)$ be the coefficient ring for $ku$-theory. It follows from the Atiyah-Hirzebruch spectral sequence that $ku_\ast(BU(1))$ is a free $\mathbb{Z}[u]$-module on the generators $1$ and $y_n$ for $n\geq 1$ (cf. \cite{adamsstablehomotopy}).

Multiplicatively the $y_n$ satisfy certain relations which are similar to the ones in a divided polynomial algebra, but `twisted' by the Bott element. From \cite[Thm. 3.4]{wilsonhopfring} we can get the following compact description: Consider the formal power series
\[
y(t):=1+\sum_{n\geq 1} y_n t^n\in ku_\ast(BU(1))[[t]]\, .
\]
Then the relations in $ku_\ast(BU(1))$ are equivalent to the identity of power series
\begin{equation} \label{eq:ys}
y(s) y(t)= y(s+t+ust)\, .
\end{equation}
Note that ` $s+t+ust$ ' is the multiplicative formal group law for $ku$-theory. For example, (\ref{eq:ys}) yields the relations

\begin{equation} \label{eq:relation}
y_1y_{n}=(n+1)y_{n+1}+nuy_n
\end{equation}
for all $n\geq 1$, but these are not the only ones (over $\mathbb{Q}$ (\ref{eq:relation}) generate all relations).

By Theorem \ref{thm:spectrumbcomu} $ku_\ast(BU(1))$ describes the homotopy ring of $\mathbb{Z}\times B_{\com}U\sslash U$. We now have the standard homotopy fibre sequence
\begin{equation} \label{eq:fibresequencebcomu}
B_{\com}U\longrightarrow \mathbb{Z}\times B_{\com}U\sslash U\stackrel{p}{\longrightarrow} \mathbb{Z}\times BU\, .
\end{equation}
Since the basepoint of $B_{\com}U$ is fixed by the conjugation action of $U$, the map $p$ in (\ref{eq:fibresequencebcomu}) has a section $\mathbb{Z}\times BU\rightarrow \mathbb{Z}\times B_{\com}U\sslash U$, thus it is surjective on homotopy groups. Since the inclusion of the fibre is a map of additive and multiplicative $H$-spaces, this shows that $\pi_\ast(B_{\com}U)\subset \pi_\ast(\mathbb{Z}\times B_{\com}U\sslash U)$ is a subring. On homotopy groups $p$ corresponds to the map $ku_\ast(BU(1))\rightarrow ku_\ast(\textnormal{pt})\cong\mathbb{Z}[u]$ induced by $BU(1)\rightarrow \textnormal{pt}$, which sends all the $y_n$ to zero and is the identity on $u$. This proves:

\begin{corollary} \label{cor:homotopygroups}
The homotopy ring $\pi_\ast(B_{\com}U)$ is the ideal $(y_n\,|\, n\geq 1)\subset ku_\ast(BU(1))$. In particular, the homotopy groups of the space $B_{\com}U$ are as follows,
\begin{alignat*}{2}
\pi_{2n}(B_{\com}U)		&= \mathbb{Z}^n, \\
\pi_{2n+1}(B_{\com}U)	&= 0
\end{alignat*}
for all $n\geq 0$.
\end{corollary}

A similar description applies to the homotopy ring of $B_{\com}SU$, which we determine at the end of Section \ref{sec:type}.

\section{The homotopy types of $B_{\com}U$ and $B_{\com}SU$} \label{sec:type}

Recall that the space $E_{\com}G$ is defined to be the homotopy fibre of $i: B_{\com}G\rightarrow BG$. It is also the total space of the universal transitionally commutative $G$-bundle over $B_{\com}G$. The main result of the present section identifies the homotopy type of $E_{\com}G$ for $G=U$ and $SU$ (the two are equivalent, cf. Lemma \ref{lem:detfibrebcomu}).

In \cite[Thm. 4.2]{ademnilpotentktheory} it is proved that the map $i: B_{\com}U\rightarrow BU$ admits a section up to homotopy $s: BU\rightarrow B_{\com}U$, which is also an infinite loop map. As a consequence, the authors obtain a splitting of infinite loop spaces
\[
B_{\com}U\simeq BU\times E_{\com}U\, .
\]
This situation is quite different from the case of a compact Lie group $G$, where one usually gets a splitting of $B_{\com}G$ only after looping once (cf. \cite[Thm. 6.3]{ademcommutingelements}). And indeed, the construction of the section $s: BU\rightarrow B_{\com}U$ makes explicit use of the loop space structure on $B_{\com}U$. \bigskip

In order to understand the map $i: B_{\com}U\rightarrow BU$ on the spectrum level we use an auxiliary spectrum $F$. Recall from Definition \ref{def:cosimplicialobject} the cosimplicial group $F_\ast$. Define $F:=|\kdef(F_\ast)|$ and let $\iota: E\rightarrow F$ be the map induced by the morphism of cosimplicial groups $F_\ast\rightarrow \mathbb{Z}^\ast$. Then $\iota$ realises the map of spaces $i: B_{\com}U\rightarrow BU$ as a map of $ku$-algebras. Indeed, $\Omega^\infty F\simeq \mathbb{Z}\times BU\sslash U$ and the map $\Omega^\infty \iota: \Omega^\infty E\rightarrow \Omega^\infty F$ is the extension of $i$ over the homotopy orbit. There is a homotopy fibre sequence
\begin{equation} \label{eq:ffibrationsequence}
BU\rightarrow \mathbb{Z}\times BU\sslash U \rightarrow \mathbb{Z}\times BU\, ,
\end{equation}
which is split by a map (of $E_{\infty}$-ring spaces) $r: \mathbb{Z}\times BU\rightarrow \mathbb{Z}\times BU\sslash U$ using the $U$-fixed basepoint of $BU$. Let $x\in \pi_2(F)$ and $u\in \pi_2(F)$ be the classes determined by the Bott element of the fibre respectively the base of this homotopy fibre sequence.

\begin{lemma} \label{lem:homotopyringf}
The homotopy ring of $F$ is $\pi_\ast(F)\cong \mathbb{Z}[u,x]/(x^2-ux)$.
\end{lemma}
\begin{proof}
The following fact will be useful. Let $G$ be a group acting on itself via conjugation. This action induces a simplicial action of $G$ on the nerve $N_\ast G$, hence on $BG$, whose homotopy orbit $BG\sslash G$ can be regarded as the classifying space of the semi-direct product $G\rtimes G$. The shear map $G\rtimes G\rightarrow G\times G$ sending $(g,h)\mapsto (gh,h)$ gives an isomorphism with the direct product, thus showing that $BG\sslash G \cong BG\times BG$.

In the special case $G=U(n)$ the shear isomorphism commutes with direct sum and tensor product. Thus, the induced homeomorphism
\[
\coprod_{n\geq 0} BU(n)\sslash U(n)\xrightarrow{\cong} \coprod_{n\geq 0} BU(n)\times BU(n)
\]
is compatible with the additive and multiplicative $H$-spaces structures on both the domain and target, where on the target space the $H$-space structures are the ones defined factorwise. After group-completion we obtain a homotopy equivalence $\sigma: \mathbb{Z}\times BU\sslash U\xrightarrow{\sim} \mathbb{Z}\times BU\times BU$, which is in addition a map of $H$-rings ($=$ a ring in the homotopy category). The homotopy ring of $\mathbb{Z}\times BU\times BU$ is isomorphic to $\mathbb{Z}[b,c]/(bc)$, where $b$ and $c$ are the Bott classes corresponding to the two factors of $BU$.

Let $l: BU\rightarrow \mathbb{Z}\times BU\sslash U$ be the inclusion of the fibre, i.e. the first arrow in (\ref{eq:ffibrationsequence}). Restricting to the components of the basepoints we obtain a homotopy commutative diagram
\[
\xymatrix{
BU\times BU \ar[r]^-{r\times l} \ar[d]^-{\Delta \times id}	& \,BU\sslash U\times BU\sslash U \ar[r]^-{\oplus}	& BU\sslash U \ar[d]^-{\sigma}_{\simeq} \\
BU\times BU\times BU \ar[rr]^-{id\times \oplus}			&										& BU\times BU\, .
}
\]
We write $\sigma_\ast$ for the map induced by $\sigma$ on homotopy groups. The diagram shows that $\sigma_\ast (u)=b+c$ and $\sigma_\ast (x)=c$. Therefore, $\sigma_\ast(ux)=\sigma_\ast(x^2)=c^2$, so $ux=x^2$. We also see that as a ring $\pi_\ast(F)$ is generated by $u$ and $x$. From counting ranks we see that $ux=x^2$ is the only relation and therefore $\pi_\ast(F)\cong \mathbb{Z}[u,x]/(ux-x^2)$ as claimed.
\end{proof}

\begin{remark} \label{rem:splittingmap}
If $X_\ast$ is a simplicial space, we write $F_k|X_\ast|\subset |X_\ast|$ for the simplicial $k$-skeleton of $|X_\ast|$, that is, for the image of $\coprod_{n\leq k} X_n\times \Delta^n$ in the geometric realisation. We can apply the same construction to the terms in a simplicial spectrum. In \cite{lawsonbottcofibre} Lawson determines the unitary deformation $K$-theory of free groups. He finds that $\kdef(F_k)\simeq ku\vee \bigvee^{k} \Sigma ku$ as $ku$-modules. From this one can deduce that the inclusion of the simplicial $1$-skeleton $F_1|\kdef(F_\ast)|\rightarrow |\kdef(F_\ast)|=F$ is an equivalence of $ku$-modules (e.g. by comparing spectral sequences). As $F_1|\kdef(F_\ast)|=F_1|\kdef(\mathbb{Z}^\ast)|$ this inclusion factors through the commutative $K$-theory spectrum $E$. The resulting $ku$-module map $F\simeq F_1|\kdef(F_\ast)|\rightarrow E$ can be seen as a spectrum level version of the splitting map $s: BU\rightarrow B_{\com}U$ constructed in \cite{ademnilpotentktheory}.
\end{remark}

\begin{lemma} \label{lem:kumodulesplitting}
There is a diagram of $ku$-modules
\[
\xymatrix{
\myvee_{n\geq 0} \Sigma^{2n}ku \ar[r]^-{f}_-{\simeq} \ar[d] & E \ar[d]^-{\iota} \\
ku\vee \Sigma^2 ku \ar[r]^-{f'}_-{\simeq} & F 
}
\]
commuting up to homotopy, where the two horizontal maps are stable equivalences and the left vertical map collapses the summands $\Sigma^{2n}ku$ for $n\geq 2$.
\end{lemma}
\begin{proof}
The splitting of $ku\wedge BU(1)_+$ as a wedge of suspensions of $ku$ is well known. Recall that $\pi_\ast(E)\cong ku_\ast(BU(1))$ is a free $\mathbb{Z}[u]$-module on generators $y_0:=1$ and $y_n$ for $n\geq 1$. Using the structure of $E$ as a $ku$-module we can define for every $n\geq 0$ a $ku$-module map $f_n: \Sigma^{2n}ku\rightarrow E$ to be the composite
\[
S^{2n}\wedge ku\xrightarrow{y_n\wedge id} E\wedge ku\xrightarrow{\textnormal{mult.}} E\, .
\]
The coproduct over the $f_n$ defines the top horizontal map in the diagram,
\[
\bigvee_{n\geq 0}f_n: \bigvee_{n\geq 0} \Sigma^{2n}ku\rightarrow E\,,
\]
which is a $\pi_\ast$-isomorphism by construction, thus a stable equivalence. The equivalence $f': ku\vee \Sigma^2 ku\rightarrow F$ is defined similarly, using the homotopy class $x\in \pi_2(F)$ and the structure of the homotopy ring described in Lemma \ref{lem:homotopyringf}. Let us write $\iota_\ast$ for the map induced by $\iota$ on homotopy groups. To see that the diagram commutes up to homotopy, we first note that $\iota_\ast (y_1)=x$, by definition of $y_1$, $x$ and $\iota$. Using Lemma \ref{lem:homotopyringf} and inducting over the multiplicative relations in (\ref{eq:relation}) we see that $\iota_\ast (y_n)=0$ for all $n\geq 2$. This shows that $\iota\circ  f$ factors, up to homotopy, through $f'$ in the way displayed in the diagram.
\end{proof}

Let $BU\langle 2n\rangle\rightarrow BU$ denote the $2n-1$ connected cover of $BU$, that is, $\pi_{\ast}(BU\langle 2n\rangle)=0$ for $\ast\leq 2n-1$ and the map to $BU$ induces an isomorphism $\pi_\ast(BU\langle 2n\rangle)\cong \pi_\ast(BU)$ for $\ast \geq 2n$. The main result of this section is:
\begin{theorem} \label{thm:splitting}
There are splittings of $E_{\infty}-\mathbb{Z}\times BU$-modules $B_{\com}U \simeq BU \times E_{\com}U$, and $B_{\com}SU \simeq BSU \times E_{\com}U$, and
\[
E_{\com}U \simeq \prod_{n\geq 2} BU\langle 2n\rangle\, .
\]
\end{theorem}
\begin{proof}
Let $\eta: ku\rightarrow E$ be the $ku$-algebra unit. The $ku$-module spectrum $b_{\com}u:=\textnormal{hcofib}(\eta)$ satisfies $\Omega^\infty b_{\com}u\simeq B_{\com}U$. Moreover, by  Lemma \ref{lem:kumodulesplitting} there is a splitting of $ku$-modules $b_{\com}u\simeq \bigvee_{n\geq 1}\Sigma^{2n}ku$ so that the map $i: B_{\com}U\rightarrow BU$ corresponds to the projection onto the $\Sigma^2ku$-summand. The infinite loop space associated to $\Sigma^{2n}ku$ is the $2n$-th term in an $\Omega$-spectrum for $ku$, i.e. the $2n-1$-connected cover $BU\langle 2n\rangle$ of $BU$. Thus, applying $\Omega^\infty$ we obtain a splitting of $E_\infty-\mathbb{Z}\times BU$-modules
\[
B_{\com}U\simeq \prod_{n\geq 1} BU\langle 2n\rangle\,,
\]
so that $E_{\com}U= \textnormal{hofib}(i)$ corresponds to the factors with $n\geq 2$.

The splitting for $B_{\com}SU$ is obtained in a similar way. Modifying the constructions in Section \ref{sec:spectrum} we obtain a $ku$-algebra $E'$ with $\Omega^\infty E'\simeq \mathbb{Z}\times B_{\com}SU\sslash U$ and a morphism of $ku$-algebras $E'\rightarrow E$ induced by the inclusion $B_{\com}SU\subset B_{\com}U$. Corollary \ref{cor:homotopybcomsu} below shows that $\pi_\ast(E')$ is a $\mathbb{Z}[u]$-submodule of $\pi_\ast(E)$ generated by $1$, $uy_1$ and $y_n$ for $n\geq 2$. 
\end{proof}

Lemma \ref{lem:detfibrebcomu} and Corollary \ref{cor:homotopybcomsu} complete the proof of Theorem \ref{thm:splitting}. Let $j: SU(n)\rightarrow U(n)$ be the inclusion and let $\textnormal{det}: U(n)\rightarrow U(1)$ be the determinant map.

\begin{lemma} \label{lem:detfibrebcomu}
For every $1\leq n\leq \infty$ there is a homotopy fibre sequence
\[
B_{\com}SU(n) \xrightarrow{B_{\com}(j)}	B_{\com}U(n) \xrightarrow{B_{\com}(\det)}  BU(1)\, .
\]
For $n=\infty$, $B_{\com}(j)$ is a map of non-unital $E_{\infty}$-ring spaces and $B_{\com}(\det): B_{\com}U_{\oplus}\rightarrow BU(1)$ is an infinite loop map.
\end{lemma}
\begin{proof}
Fix an integer $n\geq 1$. We consider the $n$-sheeted covering map
\begin{alignat*}{2}
q: U(1)\times SU(n)	&	\longrightarrow U(n)	\\
(z,A)				&	\longmapsto	zA
\end{alignat*}
with covering group the cyclic group $\mathbb{Z}/n\mathbb{Z}$. Let $k\geq 1$. Applying the functor $\Hom(\mathbb{Z}^k,-)$ gives a sequence of maps
\[
\Hom(\mathbb{Z}^k,\mathbb{Z}/n\mathbb{Z})\longrightarrow \Hom(\mathbb{Z}^k,U(1)\times SU(n))\xrightarrow{\Hom(\mathbb{Z}^k,q)} \Hom(\mathbb{Z}^k,U(n))\, .
\]
Since $\Hom(\mathbb{Z}^k,U(n))$ is path-connected for all $k$ and $n$, by \cite[Cor. 2.4]{ademhomomorphisms}, a result of Goldman \cite[Lem. 2.2]{goldmancomponents} shows that $\Hom(\mathbb{Z}^k,q)$ is a covering map with covering group $\Hom(\mathbb{Z}^k,\mathbb{Z}/n\mathbb{Z})\cong (\mathbb{Z}/n\mathbb{Z})^k$. The resulting covering sequence fits into a commutative diagram
\[
\xymatrix{
(\mathbb{Z}/n\mathbb{Z})^k \ar[r] \ar@{=}[d]	&	\,U(1)^k\times \Hom(\mathbb{Z}^k,SU(n)) \ar[rr]^-{\Hom(\mathbb{Z}^k,q)} \ar[d]^-{\textnormal{pr}_1}	&&	\Hom(\mathbb{Z}^k,U(n)) \ar[d]^-{\Hom(\mathbb{Z}^k,\det)}	\\
(\mathbb{Z}/n\mathbb{Z})^k \ar[r]			&	U(1)^k \ar[rr]^-{p_n^{\times k}}				&&	U(1)^k\, .
}
\]
The bottom row is the $k$-fold cartesian product of the covering map $p_n: U(1)\rightarrow U(1)$ sending $z\mapsto z^n$. Since both rows are homotopy fibre sequences, the right hand square is homotopy cartesian. Taking vertical homotopy fibres yields a homotopy fibre sequence
\[
\Hom(\mathbb{Z}^k,SU(n)) \xrightarrow{\Hom(\mathbb{Z}^k,j)}	\Hom(\mathbb{Z}^k,U(n)) \xrightarrow{\Hom(\mathbb{Z}^k,\det)}	U(1)^k\, .
\]
Each term in the sequence forms a levelwise path-connected simplicial space when $k$ varies. The theorem of Bousfield-Friedlander \cite[Thm. B.4]{bousfieldfriedlanderhomotopytheory} implies now that
\begin{equation} \label{eq:homotopyfibredet}
B_{\com}SU(n) \xrightarrow{B_{\com}(j)}	B_{\com}U(n) \xrightarrow{B_{\com}(\det)}	BU(1)
\end{equation}
is a homotopy fibre sequence. Here we used the fact that the simplicial spaces are good, so that we can replace them up to weak equivalence by the diagonals of bisimplicial sets. This proves the lemma for $1\leq n< \infty$. The maps in (\ref{eq:homotopyfibredet}) are natural with respect to the standard maps $B_{\com}SU(n)\rightarrow B_{\com}SU(n+1)$ and $B_{\com}U(n)\rightarrow B_{\com}U(n+1)$. Passing to homotopy colimits as $n\to \infty$ proves the case $n=\infty$. The second statement is clear, because the inclusion maps $SU(n)\hookrightarrow U(n)$ are compatible with block sum and tensor product.
\end{proof}

The determinant map $B_{\com}(\det): B_{\com}U\rightarrow BU(1)$ factors through $BU$ and induces an isomorphism on second homotopy groups. From the long exact sequence of homotopy groups for $B_{\com}SU\rightarrow B_{\com}U\rightarrow BU(1)$ we see that $\pi_\ast(B_{\com}SU)\subset \pi_\ast(B_{\com}U)$ is a subring, only missing the class $y_1\in \pi_2(B_{\com}U)$.

\begin{corollary} \label{cor:homotopybcomsu}
The homotopy ring $\pi_\ast(B_{\com}SU)$ is the subring of $\pi_\ast(B_{\com}U)$ given by the ideal $(uy_1,\,y_n\,|\, n\geq 2)\subset ku_\ast(BU(1))$.
\end{corollary}

We finish this section with a remark about cohomology operations.

\paragraph{Operations.} There are a couple of interesting cohomology operations which can be defined as self-maps of $B_{\com}U$. Clearly, we have complex conjugation $\psi^{-1}: B_{\com}U\rightarrow B_{\com}U$, but there are also involutions $\phi^t$ induced by transposition $A\mapsto A^t$, and $\phi^{-1}$ induced by taking inverses $A\mapsto A^{-1}$, $A\in U$. All these operations are compatible with direct sum and tensor product, so they induce stable multiplicative cohomology operations on $\tilde{K}_{\com}^\ast(-)$. It is not difficult to see that they extend to operations on $E$. Under the equivalence $E\simeq ku\wedge BU(1)_+$ they can be identified with more familiar operations: Complex conjugation $\psi^{-1}$ corresponds to complex conjugation on both $ku$ and $BU(1)_+$, transposition $\phi^t$ corresponds to complex conjugation only on $ku$ and leaves $BU(1)_+$ fixed, and $\phi^{-1}$ is complex conjugation on $BU(1)_+$ and the identity on $ku$. This way one could easily compute their effect on the homotopy groups of $B_{\com}U$, but we will not use them in this paper.

More generally than $\phi^{-1}$ one has an operation $\phi^k: B_{\com}U\rightarrow B_{\com}U$ for every integer $k\in \mathbb{Z}$ induced from the $k$-th power map $A\mapsto A^k$ in the unitary groups. In terms of vector bundles it represents the operation which takes a commuting cocycle to its $k$-th power. Recall that $B_{\com}U=|\Hom(\mathbb{Z}^\ast,U)|$. The map $\phi^k$ is then the map induced by the endomorphism of $\mathbb{Z}^\ast$ which is multiplication by $k$ in every simplicial degree. It is obvious then that $\phi^k$ extends to a $ku$-algebra operation on $E$. Under the equivalence with $ku\wedge BU(1)_+$ it is the map induced by the $k$-th power map in $BU(1)_+$.

These `wanna-be' Adams operations can be quite useful. In \cite{gritschacherthesis} they were used to determine the homotopy ring $\pi_\ast(E)$ from the cofibre sequence in Corollary \ref{cor:ecofibre}. They also show that the canonical map $i: B_{\com}U\rightarrow BU$ cannot have a section which is an infinite loop map and a map of multiplicative $H$-spaces at the same time: Namely, first note that $\phi^k(y_1)=k y_1$ on $\pi_2(B_{\com}U)$. Now suppose that $s: BU\rightarrow B_{\com}U$ is a multiplicative section, then the effect of the composite $i\circ \phi^k\circ s$ on $\pi_{2n}(BU)$ is multiplication by $k^n$. By a result of Clarke \cite{clarkeselfmaps} this uniquely identifies the $H$-map $i\circ \phi^k\circ s$ as the $k$-th Adams operation $\psi^k: BU\rightarrow BU$. However, $i\circ \phi^k\circ s$ is a composite of infinite loop maps, while $\psi^k$ is unstable unless $k\in \{0,\pm 1\}$. We arrive at a contradiction. A similar argument, using the fact that Adams operations are determined by their effect on line bundles, shows that if $j: BU(1)\rightarrow B_{\com}U$ and $l: BU(1)\rightarrow BU$ are the canonical maps, then $s\circ l \not\simeq j$.

\section{On the rational (co-)homology of $B_{\com}U$} \label{sec:cohomologybcomu}

In this section we investigate the relationship between the homotopy groups of $B_{\com}U$ and certain rational characteristic classes introduced in \cite{ademcommutativity}.

\subsection{The conjugation map and cohomology of $B_{\com}U(n)$} \label{sec:cohomologybcomg}

The rational cohomology algebra of $B_{\com}U$ was first described by Adem and G\'omez \cite{ademcommutativity}. Their computation relies on a general description of the rational cohomology of $B_{\com}G$ in terms of Weyl group invariants, which we now recall from \cite[\S 6]{ademcommutingelements} and \cite[\S 7]{ademcommutativity}.

Let $G$ be a compact connected Lie group. Unlike the cases $G=U(n)$ or $G=SU(n)$ the space of homomorphisms $\Hom(\mathbb{Z}^k,G)$ need not be path-connected for general $G$, see \emph{e.g.} \cite[p. 493]{ademcommutativity}. Thus let $\Hom(\mathbb{Z}^k,G)_{\mathds{1}}\subset \Hom(\mathbb{Z}^k,G)$ denote the path-component of the trivial representation. The path-components $\Hom(\mathbb{Z}^k,G)_{\mathds{1}}$ assemble into a simplicial space too, whose geometric realisation is denoted
\[
B_{\com}G_{\mathds{1}}:=|k\mapsto \Hom(\mathbb{Z}^k,G)_{\mathds{1}}|\, .
\]

Fix a maximal torus $T\leqslant G$ and let $W=N(T)/T$ be its Weyl group. An element $w=nT\in W$  acts on $gT\in G/T$ by $w\cdot gT=gn^{-1}T$ and on $t\in T$ by $w\cdot t=ntn^{-1}$. The action on $T$ induces an action on $\Hom(\mathbb{Z}^k,T)$ for every $k\geq 0$. There is a map
\begin{alignat*}{2}
G/T\times_W \Hom(\mathbb{Z}^k,T)	&\quad& \longrightarrow	&\quad \Hom(\mathbb{Z}^k,G)_{\mathds{1}}	\\
[gT,(t_1,\dots,t_k)]				&\quad& \longmapsto	&\quad (gt_1g^{-1},\dots,gt_kg^{-1})
\end{alignat*}
which generalises the well known conjugation map $G/T\times T\rightarrow G$ in Lie group theory. For varying $k\geq 0$ both domain and target of this map assemble into a simplicial space (we may regard $G/T$ as a constant simplicial space), and it is easy to check that there is a welldefined map induced on the geometric realisation
\begin{equation} \label{eq:conjugationmap1}
\varphi: G/T\times_W BT\rightarrow B_{\com}G_{\mathds{1}}\, .
\end{equation}
Now \cite[Thm. 6.1]{ademcommutingelements} asserts that the `conjugation map' $\varphi$ induces an isomorphism of $\mathbb{Q}$-algebras
\begin{equation} \label{eq:bcomgcohomology}
H^\ast(B_{\com}G_{\mathds{1}},\mathbb{Q})\cong \left(H^\ast(G/T,\mathbb{Q})\otimes H^\ast(BT,\mathbb{Q})\right)^W\, .
\end{equation}
Equivalently, using the Borel presentation for the algebra $H^\ast(G/T,\mathbb{Q})$, the isomorphism (\ref{eq:bcomgcohomology}) can be written as
\begin{equation} \label{eq:bcomgcohomology2}
H^\ast(B_{\com}G_{\mathds{1}},\mathbb{Q})\cong \left(H^\ast(BT,\mathbb{Q})\otimes H^\ast(BT,\mathbb{Q})\right)^W/J\,,
\end{equation}
where $J$ is the ideal in $\left(H^\ast(BT,\mathbb{Q})\otimes H^\ast(BT,\mathbb{Q})\right)^W$ generated by the image of the positive degree elements of $H^\ast(BG,\mathbb{Q})\cong H^\ast(BT,\mathbb{Q})^W$ under the inclusion given by $x\mapsto x\otimes 1$. This is \cite[Prop. 7.1]{ademcommutativity}. \bigskip

{\noindent \emph{Notation.} For the rest of Section \ref{sec:cohomologybcomu} $H^\ast$ and $H_\ast$ will always mean (co-)homology with rational coefficients, unless stated otherwise.
}

\paragraph{The case $G=U(n)$} \cite[\S 8.1]{ademcommutativity}. Let $T(n)\leqslant U(n)$ denote the maximal torus consisting of diagonal matrices with entries in $U(1)$. The associated Weyl group is the symmetric group $\Sigma_n$ on $n$ letters which acts on $T(n)$ by permuting the diagonal. The classifying space $BT(n)$ is a product of $n$ copies of $BU(1)$, so its rational cohomology ring is a polynomial algebra $\mathbb{Q}[y_1,\dots,y_n]$ on a set of $n$ independent variables of degree two. In the following we shall frequently write $\mathbb{Q}[\mathbf{y}]$ for a polynomial algebra, where $\mathbf{y}:=\{y_1,\dots,y_n\}$. Thus, by (\ref{eq:bcomgcohomology2})
\[
H^\ast(B_{\com}U(n))\cong (\mathbb{Q}[\mathbf{x}]\otimes \mathbb{Q}[\mathbf{y}])^{\Sigma_n}/J_n\,,
\]
where $\mathbf{x}=\{x_1,\dots,x_n\}$ is another set of variables of degree two, the symmetric group acts diagonally on the tensor product by permuting the variables in $\mathbf{x}$ and $\mathbf{y}$, and $J_n$ is the ideal generated by all $e_j(\mathbf{x})\otimes 1$ for $j\geq 1$, where
\[
e_j(\mathbf{x})=\sum_{1\leq i_1<\dots < i_j \leq n} x_{i_1}\cdots x_{i_j}
\]
denotes the $j$-th elementary symmetric polynomial in $\mathbf{x}$. For $(a,b)\in \mathbb{N}^2$ let $z_{a,b,n}\in \mathbb{Q}[\mathbf{x}]\otimes \mathbb{Q}[\mathbf{y}]$ denote the $\Sigma_n$-invariant polynomial of degree $2(a+b)$ given by
\begin{equation} \label{eq:generatorzabn}
z_{a,b,n}:=x_1^ay_1^b+\dots+x_n^ay_n^b\, .
\end{equation}
Elements of this form in the polynomial ring are called \emph{multisymmetric functions}, see \emph{e.g.} \cite{vaccarinomultisymmetric}. It is shown in \cite[\S 8.1]{ademcommutativity} that $H^\ast(B_{\com}U(n))$ is generated as an algebra by the $z_{a,b,n}$ for $(a,b)\in \mathbb{N}^2$, $b\geq 1$, $2(a+b)\leq 2n$. The authors also verify that $H^\ast(B_{\com}U)\cong \lim_n H^\ast(B_{\com}U(n))$ and that there are no relations in $H^\ast(B_{\com}U)$, that is, it is generated as a \emph{polynomial} algebra by the sequences
\begin{equation} \label{eq:generatorzab}
z_{a,b}:=(z_{a,b,n})_{n\geq 1}\,,
\end{equation}
so that
\begin{equation} \label{eq:cohomologybcomu}
H^\ast(B_{\com}U)\cong \mathbb{Q}[z_{a,b}\,|\, (a,b)\in \mathbb{N}^2,\, b\geq 1]\, .
\end{equation}

\subsection{The rational Hopf ring for $B_{\com}U$}

We want to establish a precise relationship between the $K$-homology classes $y_n\in ku_{\ast}(BU(1))$ and the rational cohomology classes $z_{a,b}\in H^{\ast}(B_{\com}U)$. In order to do so, we shall exploit the full multiplicative structure of the space $B_{\com}U$ induced by direct sum and tensor product. In fact, we wish to work with the larger space $\mathbb{Z}\times B_{\com}U\sslash U$ instead of $B_{\com}U$, so we begin by analysing its homology and cohomology.

\paragraph{The cohomology ring.} We first compare $H^\ast(B_{\com}U)$ to $H^\ast(B_{\com}U\sslash U)$. Let
\[
\bar{\varphi}: U(n)/T(n)\times BT(n)\rightarrow B_{\com}U(n)
\]
be the conjugation map (\ref{eq:conjugationmap1}) before passing to $W$-orbits. It is a $U(n)$-equivariant map if we let $U(n)$ act on $U(n)/T(n)$ by left-translation, on $B_{\com}U(n)$ by conjugation, and trivially on $BT(n)$. Let
\[
\bar{\Phi}: \left(U(n)/T(n)\times BT(n)\right)\sslash U(n)\rightarrow B_{\com}U(n)\sslash U(n)
\]
be the induced map on homotopy orbits. Note that
\[
\begin{split}
\left(U(n)/T(n)\times BT(n)\right)\sslash U(n)&=\left(U(n)/T(n)\right)\sslash U(n)\times BT(n)\\&\simeq BT(n)\times BT(n)\, .
\end{split}
\]
Recall that there is a homotopy fibre sequence of the form
\[
U(n)/T(n)\stackrel{l}{\longrightarrow} BT(n)\stackrel{j}{\longrightarrow} BU(n)\, ,
\]
where $j$ is the map induced by the inclusion $T(n)\hookrightarrow U(n)$. We obtain a commutative diagram
\begin{equation} \label{dgr:htpyorbit}{\footnotesize
\xymatrix{
U(n)/T(n)\times BT(n) \ar@{=}[r] \ar[d]^-{l\times id}			& U(n)/T(n)\times BT(n) \ar[r]^-{\bar{\varphi}} \ar[d]						& B_{\com}U(n) \ar[d]				\\
BT(n)\times BT(n) \ar[r]^-{\simeq} \ar[d]^-{\textnormal{pr}_1}	& \;\left(U(n)/T(n)\times BT(n)\right)\sslash U(n) \ar[r]^-{\bar{\Phi}} \ar[d]	& B_{\com}U(n) \sslash U(n)\ar[d]	\\
BT(n)\ar[r]^-{j}										& BU(n) \ar@{=}[r]												& BU(n)\, ,
}}
\end{equation}
\bigskip

{\noindent where the unlabelled vertical maps are the inclusion of the fibre respectively the projection onto the base in the Borel construction. The Weyl group acts in the usual way on $U(n)/T(n)$ and $BT(n)$ and these actions are compatible with all maps in the diagram. Specifically, the action on $BT(n)\times BT(n)$ is the diagonal one.}

Now consider the right half of the diagram. After passing to $W$-orbits $\bar{\varphi}$ and $\bar{\Phi}$ describe a map of homotopy fibre sequences. The induced map between the fibres is the conjugation map $\varphi$ (\ref{eq:conjugationmap1}) which is a rational cohomology isomorphism by (\ref{eq:bcomgcohomology}). By comparison of spectral sequences, we see that the map $\bar{\Phi}$ induces a rational cohomology isomorphism on $W$-orbits. The horizontal composite in the centre of the diagram yields
\[
H^\ast(B_{\com}U(n)\sslash U(n))\cong (H^\ast(BT(n))\otimes H^\ast(BT(n)))^W\, .
\]
The left-most column in the diagram shows that the right-most column takes the following form on cohomology,
\begin{equation} \label{eq:cohomologyborel}
\begin{split}
H^\ast(BT(n))^W\xrightarrow{x\mapsto x\otimes 1} &\,(H^\ast(BT(n))\otimes H^\ast(BT(n)))^W\,\cdots\\&\quad\cdots \;\xrightarrow{\textnormal{proj.}} (H^\ast(BT(n))\otimes H^\ast(BT(n)))^W/J\, .
\end{split}
\end{equation}
We can pass to the colimit as $n\to \infty$ in exactly the same way as in \cite[pp. 526]{ademcommutativity} to conclude that
\begin{equation} \label{eq:cohomologybcomuu}
H^\ast(B_{\com}U\sslash U)\cong \mathbb{Q}[z_{a,b}\,|\, (a,b)\in \mathbb{N}^2,\, (a,b)\neq (0,0)]\, .
\end{equation}
The classes $z_{a,b}$ acquire the same definition as in (\ref{eq:generatorzab}). The only difference between (\ref{eq:cohomologybcomuu}) and (\ref{eq:cohomologybcomu}) is that the restriction $b\geq 1$ present in (\ref{eq:cohomologybcomu}) is now omitted, which accounts for the additional classes coming from the base space $BU$ in the Borel construction $B_{\com}U\sslash U$.

Recall that $H^\ast(BU)\cong \mathbb{Q}[\textnormal{ch}_a\,|\,a\geq 1]$ where $\textnormal{ch}_a\in H^{2a}(BU)$ is the $a$-th component of the Chern character. Define the class $z_a:=a!\,\textnormal{ch}_a$. We see from (\ref{eq:cohomologyborel}) that the sequence of maps $B_{\com}U\rightarrow B_{\com}U\sslash U \rightarrow BU$ translates into
\begin{equation} \label{eq:mapcoh1}
\begin{split}
\mathbb{Q}[z_a\,|\, a\geq 1] &\longrightarrow \mathbb{Q}[z_{a,b}\,|\, (a,b)\in \mathbb{N}^2,\, (a,b)\neq (0,0)]\,\cdots \\&\hspace{90pt} \cdots\, \longrightarrow \mathbb{Q}[z_{a,b}\,|\, (a,b)\in \mathbb{N}^2,\, b\geq 1]\,,
\end{split}
\end{equation}
where the two arrows send $z_a\mapsto z_{a,0}$ and $z_{a,b} \mapsto \begin{cases} z_{a,b}\,, & b>0  \\ 0\,, & b=0\, , \end{cases}$ respectively.

\paragraph{The Pontrjagin ring.} Now we take into account the $H$-space structure
\[
B_{\com}U\sslash U \times B_{\com}U\sslash U\stackrel{\oplus}{\longrightarrow} B_{\com}U\sslash U
\]
induced by block sum. It makes the polynomial algebra $H^\ast(B_{\com}U\sslash U)$ into a Hopf algebra.

\begin{lemma}
The polynomial generators $z_{a,b}$ in (\ref{eq:cohomologybcomuu}) are primitive.
\end{lemma}
\begin{proof}
This is proved in exactly the same way as one proves the primitivity of the components of the Chern character $\textnormal{ch}_a\in H^{\ast}(BU)$, see \emph{e.g.} \cite{kochmanadams}.
\end{proof}

Next we consider the dual Hopf algebra $H_\ast(B_{\com}U\sslash U)$. Define $\zeta_{a,b}$ to be the dual in degree $2(a+b)$ of the indecomposable element $z_{a,b}$. Then the collection $\{\zeta_{a,b}\,|\,(a,b)\in \mathbb{N}^2,\, (a,b)\neq (0,0)\}$ spans the subspace of primitives in $H_\ast(B_{\com}U\sslash U)$.

\begin{lemma} \label{lem:homologyhopfalgebra}
The Hopf algebra $H_\ast(B_{\com}U\sslash U)$ is polynomial on the primitive generators $\zeta_{a,b}$ of degree $2(a+b)$ for $(a,b)\in \mathbb{N}^2-\{(0,0)\}$.
\end{lemma}
\begin{proof}
Since $B_{\com}U\sslash U$ is path-connected and a homotopy commutative $H$-space, its rational Pontrjagin algebra is the free graded commutative algebra on the subspace of primitives, by \cite[\S A]{milnormoorehopfalgebras}.
\end{proof}

\begin{remark} \label{rem:homologyhopfalgebra2}
We could have formulated Lemma \ref{lem:homologyhopfalgebra} for the $H$-space $B_{\com}U$ rather than the homotopy orbit $B_{\com}U\sslash U$. From the fact that the inclusion $B_{\com}U\rightarrow B_{\com}U\sslash U$ is a map of $H$-spaces and by dualising (\ref{eq:mapcoh1}) we infer that the map of Pontrjagin rings $H_\ast(B_{\com}U)\rightarrow H_{\ast}(B_{\com}U\sslash U)$ is the embedding onto the subring generated by $\zeta_{a,b}$ for $b\geq 1$.
\end{remark}

\paragraph{The Hopf ring.} So far we have only used the additive $H$-space structure on $B_{\com}U\sslash U$. In order to further simplify the homology algebra we can take into account the second $H$-space structure
\[
B_{\com}U\sslash U \times B_{\com}U\sslash U\stackrel{\otimes}{\longrightarrow} B_{\com}U\sslash U
\]
coming from the tensor product in the unitary groups. This leads naturally to the concept of a \emph{Hopf ring}. The original reference for this notion is Ravenel and Wilson \cite{wilsonhopfring}, but see also the exposition \cite{wilsonhopfringalgtop}.

Suppose that $X$ is a space equipped with a group-like `additive' $H$-space product $\oplus: X\times X\rightarrow X$ and a `multiplicative' one $\otimes: X\times X\rightarrow X$ satisfying the axioms of a ring up to homotopy. In our case, $X$ will be the space $\mathbb{Z}\times B_{\com}U\sslash U$. The additive structure induces the graded Pontrjagin product on $H_\ast(X)$. Using the diagonal map $\Delta: X\rightarrow X\times X$ the Pontrjagin algebra $H_\ast(X)$ becomes a Hopf algebra. The multiplicative $H$-space structure on $X$ induces an additional graded product on $H_\ast(X)$ which is denoted by the symbol $\circ$. The Hopf algebra $H_\ast(X)$ together with the $\circ$-product is called a Hopf ring. More abstractly, a Hopf ring is a ring object in the category of coalgebras (while a Hopf algebra would be a group object in the same category). \bigskip

Let $[n]\in H_0(\mathbb{Z})$ be the homology class determined by $n\in \mathbb{Z}$.
\begin{theorem} \label{thm:hopfring}
Rationally, the Hopf ring $H_{\ast}(\mathbb{Z}\times B_{\com}U\sslash U)$ is generated by $[1]\otimes 1$ and the classes $[0]\otimes \zeta_{1,0}$ and $[0]\otimes \zeta_{0,1}$ in degree two. In this ring the class $\zeta_{a,b}:=[0]\otimes \zeta_{a,b}$ has the presentation
\[
\zeta_{a,b}=\frac{\zeta_{1,0}^{\circ a}\circ \zeta_{0,1}^{\circ b}}{a!b!}\,,\quad (a,b)\in \mathbb{N}^2-\{(0,0)\}\, .
\]
\end{theorem}

We point out that the purely algebraic structure of the Hopf ring can be easily obtained from the homotopy ring $\pi_\ast(E)\otimes \mathbb{Q}$. Instead, our intention in proving this result is to obtain a presentation of each class $\zeta_{a,b}$ of a simple form as displayed in the theorem. The proof of the theorem is a straightforward computation, and we shall present it in Section \ref{sec:proofhopfring}. We conclude this section with a sequence of corollaries.

\begin{remark} \label{rem:hopfring2}
The inclusion $B_{\com}U\rightarrow \mathbb{Z}\times B_{\com}U\sslash U$ is a map of additive and multiplicative $H$-spaces, so by Remark \ref{rem:homologyhopfalgebra2} we can regard $\tilde{H}_\ast(B_{\com}U)$ as the ideal in the Hopf ring generated by $\zeta_{0,1}:=[0]\otimes \zeta_{0,1}$.
\end{remark}

\paragraph{The Hurewicz homomorphism.} The Hurewicz map gives a direct relationship between the $K$-homology classes $y_n$ and the cohomology classes $z_{a,b}$. Recall that the space $\mathbb{Z}\times B_{\com}U\sslash U$ has two Bott elements denoted $u$ and $y_1$ which live in the second homotopy group of the base space and the fibre, respectively, of the split fibration sequence
\[
B_{\com}U\rightarrow \mathbb{Z}\times B_{\com}U\sslash U\rightarrow \mathbb{Z}\times BU\, .
\]
It is readily checked that the Hurewicz homomorphism in dimension $2$ takes $u$ to $\zeta_{1,0}$ and $y_1$ to $\zeta_{0,1}$ (cf. Remark \ref{rem:homologyhopfalgebra2}).

\begin{corollary} \label{cor:hurewicz}
The Hurewicz homomorphism $h: \pi_\ast(B_{\com}U)\rightarrow \tilde{H}_\ast(B_{\com}U)$ is determined by the formula
\[
h(y_n)=\sum_{j=0}^{n-1} \frac{s(n,n-j)}{\binom{n}{j}} \,\zeta_{j,n-j}\,, \quad n\geq 1\,,
\]
where the $s(n,n-j)$ are the Stirling numbers of the first kind.
\end{corollary}
\begin{proof}
By construction, the Hurewicz map is multiplicative for the graded ring structure on homotopy groups and the $\circ$-product on homology. Using (\ref{eq:relation}) we can present the homotopy class $y_n$ as
\[
n!\,y_n= \prod_{j=0}^{n-1}(y_1-ju)
\]
and, therefore,
\[
n!\, h(y_n)=\circprod_{j=0}^{n-1}(\zeta_{0,1}-j \zeta_{1,0})=\sum_{j=0}^{n-1} s(n,n-j)\,\zeta_{0,1}^{\circ (n-j)}\circ \zeta_{1,0}^{\circ j}\,,
\]
by definition of the Stirling numbers. Using the relation in Theorem \ref{thm:hopfring} and dividing by $n!$ yields the desired formula.
\end{proof}

\paragraph{The splitting of $B_{\com}U$ on cohomology.} Our next aim is to describe the equivalence in Theorem \ref{thm:splitting} on rational cohomology groups. Let us write $t_k: BU\langle 2k\rangle\rightarrow BU$ for the canonical map. Let $\pi_k: B_{\com}U\rightarrow BU\langle 2k\rangle$ denote the composite of the equivalence in Theorem \ref{thm:splitting} with the projection onto the $k$-th factor. Finally, recall that $H^\ast(BU)\cong \mathbb{Q}[z_n\,|\, n\geq 1]$, where $z_n=n!\,\textnormal{ch}_n$ is a multiple of the $n$-th component of the Chern character.

\begin{corollary} \label{cor:splittingoncohomology}
For all $k\geq 1$ and all $n\geq 1$ the following formula holds
\[
(t_k\circ \pi_k)^\ast(z_n)=k! \,\sum_{j=0}^{n-1} \binom{n}{j}\,S(n-j,k)\, z_{j,n-j}\, ,
\]
where the $S(n-j,k)$ are the Stirling numbers of the second kind.
\end{corollary}
\begin{proof}
We use the commuting diagram
\[
\xymatrix{
\pi_\ast(B_{\com}U) \ar[r]^{\pi_k} \ar[d]^-{h}		& \,\pi_\ast(BU\langle 2k\rangle) \ar[d]^-{h} \ar[r]^-{t_k}	& \pi_\ast(BU) \ar[d]^-{h}	\\
\tilde{H}_{\ast}(B_{\com}U) \ar[r]^-{(\pi_k)_\ast}		& \,\tilde{H}_{\ast}(BU\langle 2k\rangle) \ar[r]^-{(t_k)_\ast}		& \tilde{H}_\ast(BU)
}
\]
to compute the map induced by $t_k\circ \pi_k$ on homology and then dualize to obtain the displayed formula for cohomology. The vertical maps in the diagram are the corresponding Hurewicz maps.

Let $\zeta_{a,b}\in \tilde{H}_\ast(B_{\com}U)$. As the splitting in Theorem \ref{thm:splitting} is a splitting of $BU$-modules, the map induced by $t_k\circ \pi_k$ on homology is a map of $\tilde{H}_\ast(BU)$-Hopf modules. The Hopf ring $\tilde{H}_\ast(BU)$ can be seen as the subring of $H_\ast(\mathbb{Z}\times B_{\com}U\sslash U)$ generated by $[0]\otimes \zeta_{1,0}$. Thus, if we decompose $\zeta_{a,b}$ using Theorem \ref{thm:hopfring}, we obtain
\[
(t_k\circ \pi_k)_\ast (\zeta_{a,b})= \frac{\zeta_{1,0}^{\circ a}}{a!b!}\circ (t_k\circ \pi_k)_\ast(\zeta_{0,1}^{\circ b})\,.
\]
As $h(y_1)=\zeta_{0,1}$ and the Hurewicz map is multiplicative, we have that $h(y_1^b)=\zeta_{0,1}^{\circ b}$. We can expand $y_1^b$ in the basis $\{u^{b-k}y_k\,|\, 1\leq k\leq b\}$ for $\pi_{2b}(B_{\com}U)$. Let $c_{b,k}\in \mathbb{Z}$ denote the coefficient in front of $u^{b-k} y_k$. Then
\[
(t_k\circ \pi_k)_\ast (\zeta_{a,b})=\frac{c_{b,k}}{a!b!} \,\zeta_{1,0}^{\circ a}\circ h(u^b)=\frac{c_{b,k}}{a!b!} \,\zeta_{1,0}^{\circ a}\circ \zeta_{1,0}^{\circ b}=c_{b,k}\binom{a+b}{a} \, \zeta_{a+b,0}\, ,
\]
where $\zeta_{a+b,0}\in H_\ast(BU)$ is dual to $z_{a+b}\in H^\ast(BU)$.

It remains to determine the coefficient $c_{b,k}$. Recall from (\ref{eq:ys}) that the multiplicative structure of $\pi_\ast(B_{\com}U)$ is encoded in the identity of power series $y(x_1)y(x_2)=y(x_1+x_2+ux_1x_2)$. Let $e_j$ denote the $j$-th elementary symmetric polynomial. By iteration,
\[
 \prod_{j=1}^b y(x_j)=y\left(\sum_{j=1}^b u^{j-1} e_j(x_1,\dots,x_b)\right)\,,
\]
so that $c_{b,k}$ is the coefficient in front of $x_1\cdots x_b$ in an expansion of
\[
 \left(\sum_{j=1}^b u^{j-1} e_j(x_1,\dots,x_b)\right)^k\,.
\]
This coefficient is counting $k!$ times the number of ways to partition the set $\{x_1,\dots,x_b\}$ into $k$ disjoint non-empty subsets, that is, $c_{b,k}=k! S(b,k)$, by definition of the Stirling number.
\end{proof}

\begin{example} \label{ex:inclusionbcomucoh}
In the case $k=1$ Corollary \ref{cor:splittingoncohomology} describes the effect of the inclusion map $i: B_{\com}U\rightarrow BU$ on cohomology. We get
\[
i^\ast(z_n)=\sum_{j=0}^{n-1}\binom{n}{j} z_{j,n-j}\, ,
\]
using the fact that $S(n-j,1)=1$. This formula can also be obtained in a different way. See Appendix \ref{sec:modulestructure}, specifically Example \ref{ex:cohomologybcomun}.
\end{example}

Our next result describes the classes $z_{0,k}$. Recall that the space $BU\langle 2k+2\rangle$ is obtained from $BU\langle 2k\rangle$ by killing the $2k$-th homotopy group, i.e., it is obtained as the homotopy fibre of a certain map
\[
\tilde{c}_k: BU\langle 2k\rangle \longrightarrow K(\mathbb{Z},2k)\,.
\]
The map $\tilde{c}_k$ represents a fractional Chern class. Suppose that $c_k: BU\rightarrow K(\mathbb{Z},2k)$ represents the $k$-th Chern class, then the effect of the composite map $c_k\circ t_k: BU\langle 2k\rangle \rightarrow K(\mathbb{Z},2k)$ on homotopy groups in degree $2k$ is multiplication by the integer $c_k(u^k)=(-1)^{k-1}(k-1)!$ (the $k$-th Chern class of the $k$-th power of the Bott element). This shows that the homotopy class of $c_k\circ t_k$ is divisible by $(k-1)!$ and we can take $\tilde{c}_k$ to be $\tilde{c}_k=(-1)^{k-1}c_k/(k-1)!$.

\begin{definition} \label{def:lambda}
There are integral characteristic classes $\lambda_k\in H^{2k}(B_{\com}U,\mathbb{Z})$ determined by the homotopy class of the composite map $\tilde{c}_k\circ \pi_k$ in the diagram
\[
\xymatrix{
B_{\com}U \ar[r]^-{\pi_k} \ar@/^2pc/[rr]^-{\lambda_k}	& BU\langle 2k\rangle \ar[r]^-{\tilde{c}_k} \ar[d]^-{t_k}		& K(\mathbb{Z},2k) \ar[d]^-{\mathbb{Z}\,\rightarrow \,\mathbb{Q}}	\\
										& BU  \ar[r]^-{\textnormal{ch}_k}					& K(\mathbb{Q},2k)\,.
}
\]
\end{definition}
The square homotopy commutes, because in rational cohomology the map $(t_k)^\ast: H^\ast(BU)\rightarrow H^\ast(BU\langle 2k\rangle)$ is the quotient map by the ideal generated by the Chern classes $c_j$ for all $j< k$ and $\tilde{c}_k$ is the class represented by $\textnormal{ch}_k$ in this quotient. \bigskip

Let $\lambda_k^{\mathbb{Q}}$ be the image of $\lambda_k$ in $H^{2k}(B_{\com}U,\mathbb{Q})$ under the map induced by the coefficient homomorphism $\mathbb{Z}\rightarrow \mathbb{Q}$.

\begin{corollary} \label{cor:lambda}
For all $k\geq 1$ we have $z_{0,k}=\lambda_k^{\mathbb{Q}}$ in $H^{2k}(B_{\com}U,\mathbb{Q})$, hence the classes $z_{0,k}$ are defined in integral cohomology.
\end{corollary}
\begin{proof}
The diagram shows that $\lambda_k^{\mathbb{Q}}$ is determined by the homotopy class of $\textnormal{ch}_k\circ t_k\circ \pi_k$. By Corollary \ref{cor:splittingoncohomology}, this is the cohomology class
\[
\lambda_k^{\mathbb{Q}}=(t_k\circ \pi_k)^\ast(\textnormal{ch}_k)=(t_k\circ \pi_k)^\ast(z_k/k!)=z_{0,k}\,,
\]
where we used the fact that $S(k-j,k)=0$ for $j>0$ and $S(k,k)=1$.
\end{proof}

\begin{remark}
In terms of spectra, the characteristic classes $\lambda_k$ are given by the components of the splitting $E\simeq \bigvee_{n\geq 0}\Sigma^{2n}ku$ and then smashing over $ku$ with $H\mathbb{Z}$. In other words, they are given by the components of the map $\lambda: E\rightarrow H\mathbb{Z}\wedge BU(1)_+\simeq \bigvee_{n\geq 0} \Sigma^{2n}H\mathbb{Z}$ which appears in Corollary \ref{cor:ecofibre}. Thus they come from diagonalizing unitary representations of free abelian groups or, in more geometric terms, they are related to the eigenvalues of the transition functions of a transitionally commutative vector bundle.
\end{remark}

\paragraph{Comparison to the cohomology of $B_{\com}SU(2)$.} The formula in Corollary \ref{cor:splittingoncohomology} also tells us something about integral cohomology. We will illustrate this in an example, which relates to work in preparation with Antol\'in Camarena and Villarreal \cite{antolinbcomu2}. In \cite[Ex. 6.4]{ademcommutativity} Adem and G\'omez show that $H^{k}(B_{\com}SU(2),\mathbb{Z})\cong \mathbb{Z}/2\mathbb{Z}$ for all $k>2$ with $k\equiv 2\textnormal{ mod } 4$, and the authors were curious about the origin of this torsion (note that $H^\ast(BSU(2),\mathbb{Z})$ is torsionfree!). The purpose of this paragraph is to show (cf. Corollary \ref{cor:z2stable}) that these torsion classes are related to the stable cohomology class $\lambda_3$, that is, to the fractional Chern class
\[
B_{\com}SU(2)\xrightarrow{\textnormal{incl.}} B_{\com}U \xrightarrow{\pi_3} BU\langle 6\rangle \xrightarrow{c_3/2} K(\mathbb{Z},6)\, .
\]

To explain this, we begin by describing the canonical maps $i_{SU(2)}: B_{\com}SU(2)\rightarrow B_{\com}U$ and $i_{U(2)}: B_{\com}U(2)\rightarrow B_{\com}U$ induced by the inclusions of groups $SU(2)\hookrightarrow U$ and $U(2)\hookrightarrow U$ on integral cohomology. In \cite{villarrealthesis} (see also \cite{antolinbcomu2}) it is shown that the ring $H^\ast(B_{\com}U(2),\mathbb{Z})$ admits a presentation with four generators:\footnote{In \cite{villarrealthesis} the classes $d_1$, $d_2$ are denoted $y_1$, $y_2$, respectively. Unfortunately, we have already reserved these names for elements in the homotopy groups of $B_{\com}U$.}
\begin{equation} \label{eq:cohbcomu2}
H^\ast(B_{\com}U(2),\mathbb{Z})\cong \mathbb{Z}[c_1,c_2,d_1,d_2]/(2d_2-c_1d_1,\, d_1^2,\, d_1d_2,\, d_2^2)\,.
\end{equation}
The generators live in degrees $\textnormal{deg}(c_i)=2i$ and $\textnormal{deg}(d_i)=2i+2$, $i=1,2$. Abusing notation, the classes $c_1$ and $c_2$ are defined to be the images of the first and second universal Chern classes living in $H^\ast(BU(2),\mathbb{Z})$ under the canonical map $i^\ast: H^{\ast}(BU(2),\mathbb{Z})\rightarrow H^\ast(B_{\com}U(2),\mathbb{Z})$. In the work cited it is also shown that
\begin{equation} \label{eq:cohbcomsu2}
H^\ast(B_{\com}SU(2),\mathbb{Z})\cong H^\ast(B_{\com}U(2),\mathbb{Z})/(c_1)\,,
\end{equation}
and that the map induced by $SU(2)\hookrightarrow U(2)$ is the projection map. From the presentation one sees that the class $d_2$ becomes $2$-torsion in $H^\ast(B_{\com}SU(2),\mathbb{Z})$ and its $c_2$-multiples give rise to the $\mathbb{Z}/2\mathbb{Z}$-summands discovered in \cite{ademcommutativity}. 

The integral cohomology groups of $B_{\com}U$ can be obtained from the equivalence in Theorem \ref{thm:splitting}. Since the cohomology of $B_{\com}U(2)$ is generated in degrees $\leq 6$, it will suffice to consider the cohomology of $BU\times BSU\times BU\langle 6\rangle$. In the relevant degrees these spaces have the following cohomology classes, each generating an infinite cyclic summand:
\begin{center}
{\def\arraystretch{1.13}\tabcolsep=5pt
\begin{tabular}{l||c|c|c|c}
deg	&	$BU$	&	$BSU$		&	$BU\langle 6\rangle$		\\	\hline\hline
$2$	&	$c_1$	&				&						\\	
$4$	&	$c_2$	&	$\lambda_2$	&						\\
$6$	&	$c_3$	&	$c_3^{SU}$	&	$\lambda_3$			
\end{tabular}
}
\end{center}
Here we use the notation $c_3^{SU}$ to distinguish the third Chern class in $BSU$ from that in $BU$, and $\lambda_2$ denotes minus the second Chern class in $BSU$ to match with our earlier notation.

\begin{lemma} \label{lem:bcomutou2}
The integral cohomology of $B_{\com}U$ in degree $\leq 6$ maps in the following way onto the cohomology of $B_{\com}SU(2)$ and $B_{\com}U(2)$:
\begin{center}
{\def\arraystretch{1.30}\tabcolsep=5pt
\begin{tabular}{l||l|l|l}
$B_{\com}U$	&	$B_{\com}U(2)$					&	$B_{\com}SU(2)$	\\	\hline\hline
$c_1$		&	$c_1$						&	$0$				\\
$c_2$		&	$c_2$						&	$c_2$		\\
$c_3$		&	$0$								&	$0$				\\
$\lambda_2$	&	$c_1^2-2c_2+d_1$			&	$-2c_2+d_1$	\\	
$c_3^{SU}$	&	$2c_1^3-6c_1c_2+8d_2$	&	$0$				\\
$\lambda_3$	&	$c_1^3-3 c_1c_2+3d_2$	&	$d_2$				
\end{tabular}
}
\end{center}
In particular, the canonical map $i_{U(2)}: B_{\com}U(2)\rightarrow B_{\com}U$ induces a surjective map of integral cohomology rings.
\end{lemma}
\begin{proof}
From the presentation we see that $H^\ast(B_{\com}U(2),\mathbb{Z})$ is torsionfree. In view of the rational isomorphism (\ref{eq:bcomgcohomology}) this implies that the conjugation map (\ref{eq:conjugationmap1}) detects the integral cohomology of $B_{\com}U(2)$, that is, the map
\begin{equation} \label{eq:conjugationmapu2}
H^\ast(B_{\com}U(2),\mathbb{Z})\rightarrow (H^\ast(U(2)/T(2),\mathbb{Q})\otimes H^\ast(BT(2),\mathbb{Q}))^{\Sigma_2}
\end{equation}
is injective. The target is generated as a $\mathbb{Q}$-algebra by the multisymmetric functions $z_{0,1}$, $z_{0,2}$, and $z_{1,1}$ defined in (\ref{eq:generatorzabn}) (here we omitted the third index from these classes, which only indicated the rank of the group). A careful inspection of the calculations in \cite[pp. 66]{villarrealthesis} (\cite{antolinbcomu2}) shows that $d_1\mapsto 2z_{1,1}$ and $d_2\mapsto z_{0,1}z_{1,1}$ under (\ref{eq:conjugationmapu2}). Moreover, either from Example \ref{ex:inclusionbcomucoh} or from (\ref{eq:ibcomun}) we see that $c_1\mapsto z_{0,1}$ and $c_2\mapsto (z_{0,1}^2-z_{0,2})/2+z_{1,1}$ under (\ref{eq:conjugationmapu2}). The effect of
\[
i_{U(2)}^\ast: H^\ast(B_{\com}U,\mathbb{Z})\rightarrow H^\ast(B_{\com}U(2),\mathbb{Z})
\]
can be described, then, by composing $i_{U(2)}^\ast$ with (\ref{eq:conjugationmapu2}) and identifying its image in terms of $z_{0,1}$, $z_{0,2}$ and $z_{1,1}$.

From Corollary \ref{cor:lambda} we see that $\lambda_2 \mapsto z_{0,2}$ and $\lambda_3 \mapsto z_{0,3}$ under this composite map. For $c_3^{SU}$ we use Corollary \ref{cor:splittingoncohomology}. First note that the map $t_2^\ast: H^\ast(BU)\rightarrow H^\ast(BSU)$ is the projection sending $c_1\mapsto 0$. Also recall that $z_3=c_1^3-3c_1c_2+3c_3$ in $H^\ast(BU)$, by definition of the class $z_3=3! \textnormal{ch}_3$. Therefore,
\[
(t_2\circ \pi_2)^\ast(z_3)=3\pi_2^\ast(c_3^{SU})=2! (S(3,2) z_{0,3}+3S(2,2) z_{1,2})\,,
\]
and, noting that $S(3,2)=3$ and $S(2,2)=1$, we get $c_3^{SU}\mapsto 2 (z_{0,3}+z_{1,2})$ under $i_{U(2)}^\ast$ followed by (\ref{eq:conjugationmapu2}). Both $z_{0,3}$ and $z_{1,2}$ can be written in terms of $z_{0,1}$, $z_{0,2}$ and $z_{1,1}$, namely $z_{0,3}=(3 z_{0,1}z_{0,2}-z_{0,1}^3)/2$ and $z_{1,2}=z_{0,1}z_{1,1}$. This way one obtains the second column in the lemma. The third column is obtained from the second by setting $c_1$ to zero and noting that $d_2$ becomes an element of order two.
\end{proof}

The entry in the bottom right corner of the table in Lemma \ref{lem:bcomutou2} shows that the $2$-torsion class $d_2$ is the image of the stable class $\lambda_3$.\bigskip

We see that for every $n\geq 2$ the class $d_2$ lifts to a class $d_{2;n} \in H^6(B_{\com}SU(n),\mathbb{Z})$, namely the image of $\lambda_3$ under $H^6(B_{\com}U,\mathbb{Z})\rightarrow H^6(B_{\com}SU(n),\mathbb{Z})$.

\begin{lemma}
For $n\geq 3$ the class $d_{2;n}$ is non-torsion.
\end{lemma}
\begin{proof}
It suffices to show that the rational class $\lambda_3^{\mathbb{Q}}$ is mapped non-trivially to $H^6(B_{\com}SU(n),\mathbb{Q})$ whenever $n\geq 3$. We can identify $\lambda_3^{\mathbb{Q}}$ with the multisymmetric function $z_{0,3}$ (Corollary \ref{cor:lambda}). In \cite[\S 8]{ademcommutativity} the authors show that the map $H^\ast(B_{\com}U(n),\mathbb{Q})\rightarrow H^\ast(B_{\com}SU(n),\mathbb{Q})$ is the quotient map by the ideal generated by $z_{0,1}$. Suppose that $n\geq 3$ and that $z_{0,3}$ is a multiple of $z_{0,1}$ in $H^\ast(B_{\com}U(n),\mathbb{Q})$. Then this lifts to a relation in the ring $(\mathbb{Q}[\mathbf{x}]\otimes \mathbb{Q}[\mathbf{y}])^{\Sigma_{n}}$ expressing $z_{0,3}$ as a polynomial in multisymmetric functions. However, it follows from \cite[Prop.~3.1(2)]{vaccarinomultisymmetric} that for $n\geq 3$ there are no non-trivial relations amongst multisymmetric functions in $(\mathbb{Q}[\mathbf{x}]\otimes \mathbb{Q}[\mathbf{y}])^{\Sigma_{n}}$ which are of degree $\leq 3$ in the variables $\mathbf{x}$ and $\mathbf{y}$. So $z_{0,3}$ is not a multiple of $z_{0,1}$ and therefore $z_{0,3}=\lambda_3^{\mathbb{Q}}$ is mapped non-trivially to $H^6(B_{\com}SU(n),\mathbb{Q})$.
\end{proof}

We summarize:

\begin{corollary} \label{cor:z2stable}
The cohomology classes in $H^\ast(B_{\com}SU(2),\mathbb{Z})$ and $H^\ast(B_{\com}U(2),\mathbb{Z})$ are stable, in the sense that they are pulled back from $H^\ast(B_{\com}U,\mathbb{Z})$. In particular, the $\mathbb{Z}/2\mathbb{Z}$-summands in $H^\ast(B_{\com}SU(2),\mathbb{Z})$ correspond to the $c_2$-multiples of $\lambda_3\in H^6(B_{\com}U,\mathbb{Z})$. For $n\geq 3$ they correspond to non-torsion classes in $H^\ast(B_{\com}SU(n),\mathbb{Z})$.
\end{corollary}

\subsection{Proof of Theorem \ref{thm:hopfring}} \label{sec:proofhopfring}

We consider the ring space up to homotopy $\mathbb{Z}\times B_{\com}U\sslash U$. For $n\in \mathbb{Z}$ we write $X[n]$ for the subspace $\{n\}\times B_{\com}U\sslash U$. The cohomology of $X[n]$ is described in (\ref{eq:cohomologybcomuu}). The multiplicative $H$-space structure restricts to component maps
\begin{equation} \label{eq:circprod}
\mu_{m,n}: X[m]\times X[n]\rightarrow X[mn]
\end{equation}
for all $m,n\in \mathbb{Z}$.

The following lemma describes the effect of (\ref{eq:circprod}) on cohomology groups.
\begin{lemma} \label{lem:mumncoh}
Let $z_{a,b}\in H^\ast(X[mn])$ with $m,n\geq \textnormal{deg}(z_{a,b})=2(a+b)$. Then
\[
\mu_{m,n}^\ast(z_{a,b})=\sum_{i=0}^a \sum_{j=0}^b \binom{a}{i}\binom{b}{j}\, z_{a-i,b-j}\otimes z_{i,j}
\]
where $z_{0,0}\otimes 1:=m$ and $1\otimes z_{0,0}:=n$.
\end{lemma}
\begin{proof}
For $k\geq 1$ let $j_k: B_{\com}U(k)\sslash U(k)\rightarrow X[k]$ denote the inclusion into the direct limit $X[k]\simeq \textnormal{colim}_{l\geq k}\, B_{\com}U(l)\sslash U(l)$. The map $j_k^\ast: H^\ast(X[k])\rightarrow H^\ast(B_{\com}U(k)\sslash U(k))$ is then the projection onto the $k$-th component in the inverse limit $H^\ast(X[k])\cong \lim_l\, H^\ast(B_{\com}U(l)\sslash U(l))$. By \cite[Prop. 3.1(2)]{vaccarinomultisymmetric} the map $j_k^\ast$ is injective in cohomological degrees $\ast\leq k$. Consider the diagram
\[
\xymatrix{
B_{\com}U(m)\sslash U(m) \times B_{\com}U(n)\sslash U(n) \ar[r]^-{\mu_{m,n}'\;} \ar[d]^-{j_m\times j_n}	& \,B_{\com}U(mn)\sslash U(mn) \ar[d]^-{j_{mn}} \\
X[m]\times X[n] \ar[r]^-{\mu_{m,n}} 														& X[mn]
}
\]
where $\mu'_{m,n}$ is also induced by tensor product. The diagram commutes up to homotopy, because the product map $\mu_{m,n}$ is induced by extending $\mu_{m,n}'$ over the group-completion.

Both vertical maps are injective in cohomological degree $2(a+b)$ by our assumption that $m,n\geq \textnormal{deg}(z_{a,b})$. Thus the formula displayed in the lemma can be checked by computing $(\mu_{m,n}')^\ast(z_{a,b,mn})$ instead, where $z_{a,b,mn}$ was defined in (\ref{eq:generatorzabn}). Because the composite in the middle row of diagram (\ref{dgr:htpyorbit}) is multiplicative with respect to tensor product, this amounts to computing the pullback of the multisymmetric function $z_{a,b,mn}$ under the map
\[
BT(m)\times BT(m)\times BT(n)\times BT(n)\xrightarrow{(\otimes\times \otimes)\circ (1\times \tau\times 1)} BT(mn)\times BT(mn)\,,
\]
where $\tau$ is the transposition of the two factors in the middle. It is readily checked that this yields the formula displayed in the lemma.
\end{proof}

Note that, even though $H^\ast(X[mn])\cong H^\ast(X[0])$, the right hand side of the above formula depends on $m,n$. We now dualise to obtain the corresponding formula for homology.
\begin{lemma} \label{lem:mumnho}
Let $\zeta_{a,b}\in H_\ast(X[m])$ and $\zeta_{c,d}\in H_\ast(X[n])$ and assume that $m,n\geq \textnormal{deg}(\zeta_{a,b})+\textnormal{deg}(\zeta_{c,d})$. Then
\[
(\mu_{m,n})_\ast(\zeta_{a,b}\otimes \zeta_{c,d})=\binom{a+c}{a}\binom{b+d}{b}\, \zeta_{a+c,b+d}+mn\, \zeta_{a,b} \zeta_{c,d}\, .
\]
\end{lemma}
\begin{proof}
Let $z^{I}$ be any monomial in the cohomology algebra of $X[mn]$ of total degree $2(a+b+c+d)$. Then
\[
\langle z^I, (\mu_{m,n})_\ast(\zeta_{a,b}\otimes \zeta_{c,d})\rangle=\langle \mu_{m,n}^\ast(z^I),\zeta_{a,b}\otimes \zeta_{c,d}\rangle
\]
is non-zero if and only if the term $z_{a,b}\otimes z_{c,d}$ appears in $\mu_{m,n}^\ast(z^I)$. By Lemma \ref{lem:mumncoh} (this uses the assumption on $m,n$) this happens if either $z^I=z_{a+c,b+d}$ or if $z^I=z_{a,b}z_{c,d}$. In the first case, the pairing evaluates to the product of binomial coefficients and, since $z_{a+c,b+d}$ is dual to $\zeta_{a+c,b+d}$, this gives the first term in the claimed formula. In the latter case we have
\[
\begin{split}
\mu_{m,n}^\ast(z_{a,b})\mu_{m,n}^\ast(z_{c,d})&=mn\,(z_{a,b}\otimes z_{c,d} + z_{c,d}\otimes z_{a,b})\\&\quad+\textnormal{terms evaluating to zero on }\zeta_{a,b}\otimes \zeta_{c,d}\, .
\end{split}
\]
If $(a,b)\neq (c,d)$, then this evaluates to $mn$ on $\zeta_{a,b}\otimes \zeta_{c,d}$, and because $z_{a,b}z_{c,d}$ is dual to $\zeta_{a,b}\zeta_{c,d}$, we are done. If $(a,b)=(c,d)$, then the pairing evaluates to $2mn$, but in this case $z_{a,b}^2$ is dual to $2\,\zeta_{a,b}^2$, and again we arrive at the displayed formula.
\end{proof}

We can now finish the proof of Theorem \ref{thm:hopfring}. We have an isomorphism of Pontrjagin algebras
\begin{equation} \label{eq:pontrzbcomu}
H_\ast(\mathbb{Z}\times B_{\com}U\sslash U)\cong \mathbb{Q}[\mathbb{Z}]\otimes_{\mathbb{Q}} \mathbb{Q}[\zeta_{a,b}\,|\, (a,b)\neq (0,0)]\,,
\end{equation}
where $\mathbb{Q}[\mathbb{Z}]$ is the rational group ring of the additive group of integers. Denote the coproduct by $\Delta$. If $[n]\in H_0(\mathbb{Z})$ denotes the homology class determined by $n\in \mathbb{Z}$, then the coproduct in $\mathbb{Q}[\mathbb{Z}]$ is given by $\Delta([n])=[n]\otimes [n]$. Furthermore, because $\mathbb{Z}$ is a ring, the group ring $\mathbb{Q}[\mathbb{Z}]$ is a Hopf ring with $[m]\circ [n]=[mn]$.

Let us write $i_n: X[n]\rightarrow \mathbb{Z}\times B_{\com}U\sslash U$ for the inclusion of the $n$-th component. Then $i_n$ sends the class $\zeta_{a,b}\in H_\ast(X[n])$ to the class $[n]\otimes \zeta_{a,b}$ under the isomorphism (\ref{eq:pontrzbcomu}).

Now let $\zeta_{a,b},\zeta_{c,d}\in H_\ast(B_{\com}U\sslash U)$ be given and choose $m,n\geq \textnormal{deg}(\zeta_{a,b})+\textnormal{deg}(\zeta_{c,d})$. The commutative diagram
\[
\xymatrix{
X[m]\times X[n] \ar[r]^-{\mu_{m,n}} \ar[d]^-{i_m\times i_n}	& X[mn] \ar[d]^-{i_{mn}}	\\
(\mathbb{Z}\times B_{\com}U\sslash U) \times (\mathbb{Z}\times B_{\com}U\sslash U) \ar[r]^-{\otimes} & \,\mathbb{Z}\times B_{\com}U\sslash U
}
\]
together with Lemma \ref{lem:mumnho} yields the relation
\[
\begin{split}
([m]\otimes \zeta_{a,b})\circ ([n]\otimes \zeta_{c,d})&=\binom{a+c}{c} \binom{b+d}{b}\, [mn]\otimes \zeta_{a+c,b+d}\\&\quad\quad\quad +mn\, [mn]\otimes (\zeta_{a,b}\zeta_{c,d})
\end{split}
\]
in $H_\ast(\mathbb{Z}\times B_{\com}U\sslash U)$. Simple formal manipulations in the Hopf ring, which we learned from \cite{wilsonhopfringalgtop}, turn this into
\[
([0]\otimes \zeta_{a,b})\circ ([0]\otimes \zeta_{c,d})=\binom{a+c}{c} \binom{b+d}{b}\, [0]\otimes \zeta_{a+c,b+d}\,.
\]
This gives us the presentation of $[0]\otimes \zeta_{a,b}$ asserted in the theorem.

\section{The commutative $K$-theory of $S^4$} \label{sec:kcoms4}

We now give a description of the group $\tilde{K}_{\com}(S^4)\cong \mathbb{Z}\oplus\mathbb{Z}$. We will show that it is generated by rank-$2$ bundles, in the sense that the natural map
\[
\pi_4(B_{\com}SU(2))\rightarrow \pi_4(B_{\com}U)
\]
is an isomorphism.

We first fix an orientation for $S^4$. Let $[S^4]\in H_4(S^4,\mathbb{Z})$ be the image of the fundamental class of $\mathbb{C}P^2$ under the quotient map $\mathbb{C}P^2\rightarrow \mathbb{C}P^2/\mathbb{C}P^1\cong S^4$. The following lemma allows us to work with cohomology.

\begin{lemma} \label{lem:pi4bcomsu2}
There is an isomorphism
\[
\pi_4(B_{\com}SU(2)) \cong \Hom(H^4(B_{\com}SU(2),\mathbb{Z}),\mathbb{Z})
\]
that takes $f\mapsto (a\mapsto \langle f^\ast(a),[S^4]\rangle)$.
\end{lemma}
\begin{proof}
This follows from Hurewicz' theorem and universal coefficients, because $B_{\com}SU(2)$ is $3$-connected and $H_4(B_{\com}SU(2),\mathbb{Z})$ is torsionfree. Both these facts were shown in \cite{ademcommutativity}.
\end{proof}

As a basis for the cohomology group $H^4(B_{\com}SU(2),\mathbb{Z})\cong \mathbb{Z}^2$ we use $c_2$ and $b:=d_1-2c_2$ in the notation of (\ref{eq:cohbcomu2}). Let $c_2^\ast$ and $b^\ast$ be their $\mathbb{Z}$-linear duals.

\paragraph{Construction of $c_2^\ast$.} We define an element $v\in \pi_4(B_{\com}SU(2))$ as follows. The simplicial $1$-skeleta of $BSU(2)$ and $B_{\com}SU(2)$ agree and are both given by $\Sigma SU(2)$. We have that $\Sigma SU(2)\cong S^4$. Thus the inclusion of the simplicial $1$-skeleton into $BSU(2)$ can be factored through a map
\[
v: S^4\cong \Sigma SU(2)\longrightarrow B_{\com}SU(2)\,,
\]
The composition $i\circ v$ is $6$-connected, because the homotopy fibre of $i\circ v: \Sigma SU(2)\rightarrow BSU(2)$ is equivalent to the simplicial $1$-skeleton of $ESU(2)$, which is equivalent to the join $SU(2)\ast SU(2)\cong S^7$. In particular, $(i\circ v)^\ast(c_2)=v^\ast(c_2)$ is a generator of $H^4(S^4,\mathbb{Z})$. By choosing the identification $S^4\cong \Sigma SU(2)$ we can arrange for $v^\ast(c_2)=1$.

We now claim that $v^\ast(b)=0$. To see this, we use the commutative diagram
\begin{equation} \label{dgr:bcomsu2}
\xymatrix{
SU(2)/T\times \Sigma T \ar[r]^-{\textnormal{incl.}} \ar[d]^-{\times 2}	& \;SU(2)/T\times BT \ar[d]^-{\bar{\varphi}}		\\
\Sigma SU(2) \ar[r]^-{v}								& B_{\com}SU(2)
}
\end{equation}
where $T\leqslant SU(2)$ is a maximal torus, and $\bar{\varphi}$ is the conjugation map (\ref{eq:conjugationmap1}) before taking Weyl group orbits. The left hand vertical arrow is the restriction of $\bar{\varphi}$ to the simplicial $1$-skeleta. It is a map of degree $2$ (the order of the Weyl group $W=\Sigma_2$). We now appeal to the proof of Lemma \ref{lem:bcomutou2} from which it follows that $d_1\mapsto 4z_{1,1}$ and $2c_2\mapsto 4z_{1,1}-2z_{0,2}$ under $\bar{\varphi}^\ast$ (the extra factor of $2$ comes from changing tori from $T(2)\leqslant U(2)$ to $T\leqslant SU(2)$). Thus, $d_1-2c_2\mapsto 0$ in the top left corner of the diagram, because $z_{0,2}$ lands in $H^4(\Sigma T,\mathbb{Q})=0$. Then also $v^\ast(b)=v^\ast(d_1-2c_2)=0$, because the left vertical map is injective on cohomology. \bigskip

The construction of $b^\ast$ seems more complicated. We expect that this class is related to an element in the second homotopy group of the space $\Hom(\mathbb{Z}^2,SU(2))/SU(2)\vee SU(2)$. For the purpose of this section, however, it is enough to define a representative for a multiple of $b^\ast$. Consider the map
\begin{equation} \label{eq:cp2tobcomsu2}
\mathbb{C}P^2\xrightarrow{\textnormal{incl.}} \mathbb{C}P^\infty\simeq BT\xrightarrow{\textnormal{incl.}} B_{\com}SU(2)\, .
\end{equation}
As noted earlier, the space $B_{\com}SU(2)$ is $3$-connected, so this composite map factors up to homotopy through the quotient space $\mathbb{C}P^2/\mathbb{C}P^1\simeq S^4$. Let $w$ denote the induced map on $S^4$. The homotopy class of $w$ is determined by that of (\ref{eq:cp2tobcomsu2}), because we are collapsing the $2$-dimensional complex $\mathbb{C}P^1$ inside a $3$-connected space. By our choice of orientation of $S^4$ the quotient map $\mathbb{C}P^2\rightarrow S^4$ has degree $1$.

To check the effect of $w$ on cohomology, we can use the composition of the map $\bar{\varphi}$ in the diagram above with the inclusion $BT\rightarrow SU(2)/T\times BT$ of the second factor. We find $w^\ast(c_2)=-1$ and $w^\ast(b)=2$. Thus, $w$ corresponds to $-c_2^\ast+2b^\ast$ under the isomorphism in Lemma \ref{lem:pi4bcomsu2}. Consequently, $v$ and $(v+w)/2$ freely generate $\pi_4(B_{\com}SU(2))$.

\begin{proposition} \label{prop:kcoms4}
The natural map $\pi_4(B_{\com}SU(2))\rightarrow \pi_4(B_{\com}U)$ induced by $SU(2)\hookrightarrow U$ is an isomorphism. As a consequence, the map $B_{\com}U(2)\rightarrow B_{\com}U$ is $4$-connected.
\end{proposition}
\begin{proof}
The proof is a computation of characteristic classes. Let us write $i_{SU(2)}: B_{\com}SU(2)\rightarrow B_{\com}U$ for the inclusion. We have the following maps,
\[
\xymatrix{
S^4 \ar[r]^-{v,\,w}	& B_{\com}SU(2) \ar[r]^-{i_{SU(2)}\;}	& B_{\com}U \ar[d]^-{i} \ar[r]^-{\lambda_2}	& K(\mathbb{Z},4) \\
					& 							& BU	\,.							&
}
\]
The group $\pi_4(BU)\cong \mathbb{Z}$ is generated by the square of the Bott class $u^2$ which has second Chern class $-1$. The group $\pi_4(B_{\com}U)\cong \mathbb{Z}\oplus \mathbb{Z}$ is generated by the classes $uy_1$ and $y_2$. The maps $i_\ast$ and $(\lambda_2)_{\ast}$ induced on homotopy groups act as the projections onto the two $\mathbb{Z}$-factors, that is, $i_\ast(uy_1)=u^2$, $(\lambda_2)_\ast(uy_1)=0$, and $i_\ast(y_2)=0$, $(\lambda_2)_\ast(y_2)=1$. As elements of $\pi_4(B_{\com}U)$ we can write $i_{SU(2)}\circ v=\alpha uy_1+\beta y_2$ for some $\alpha ,\beta\in \mathbb{Z}$. To determine $\alpha$ we compute
\[
(i\circ i_{SU(2)}\circ v)^\ast(c_2)=v^\ast(c_2)=1\,,
\]
hence $\alpha=-1$ (because $uy_1$ corresponds to $u^2$ which has $c_2(u^2)=-1$). For $\beta$ we use the table in Lemma \ref{lem:bcomutou2} which shows that
\[
(i_{SU(2)}\circ v)^\ast(\lambda_2)=v^\ast(-2c_2+d_1)=v^\ast(b)=0\,,
\]
so $\beta=0$. This shows that $i_{SU(2)}\circ v=-uy_1$ in $\pi_4(B_{\com}U)$. In a similar way one shows that $i_{SU(2)}\circ w=2y_2+uy_1$. Hence $i_{SU(2)}\circ (v+w)/2=y_2$ and the first assertion follows.

For the second statement note that both $B_{\com}U(2)$ and $B_{\com}U$ are simply connected, and that the natural map $B_{\com}U(2)\rightarrow B_{\com}U$ is an isomorphism on homotopy in dimension $2$ by Hurewicz' theorem and Lemma \ref{lem:bcomutou2}. By Lemma \ref{lem:detfibrebcomu}, the map $B_{\com}SU(2)\rightarrow B_{\com}U(2)$ induces ismorphisms on homotopy groups in dimensions $\geq 3$. The second assertion follows now, because $B_{\com}SU(2)\rightarrow B_{\com}U$ is an isomorphism on homotopy groups in dimension $3$ (trivally) and $4$ (by the first statement), and a surjection in dimension $5$ (because $\pi_5(B_{\com}U)=0$).
\end{proof}

\begin{remark}
Let $h: SU(2)/T\times \Sigma T\rightarrow B_{\com}SU(2)$ be the composite map through the top right corner of (\ref{dgr:bcomsu2}). In \cite[Ex. 2.5]{ademcommutativity} the authors construct the classifying map of a transitionally commutative $SU(2)$-bundle over $S^4$ by factoring this map through the smash product
\[
\tilde{h}: S^4\cong SU(2)/T\wedge \Sigma T\rightarrow B_{\com}SU(2)\,,
\]
using the fact that $B_{\com}SU(2)$ is $3$-connected. With the results of this section we can easily determine the class represented by $\tilde{h}$ in $\tilde{K}_{\com}(S^4)$. We have seen above that $h^\ast(c_2)=\pm 2$ and that $h^\ast(b)=0$. Thus $\tilde{h}=\pm 2v$ as an element of $\pi_4(B_{\com}SU(2))$. By Proposition \ref{prop:kcoms4}, then $\tilde{h}=\pm 2 uy_1$ in $\tilde{K}_{\com}(S^4)$ (the sign can be fixed by choosing orientations).
\end{remark}

We finish this section by explaining how the class $y_2\in \tilde{K}_{\com}(S^4)$ arises from the difference of the tautological bundle on $\mathbb{C}P^2$ thought of as a transitionally commutative bundle on the one hand, and as an ordinary line bundle on the other hand:

Recall the spectrum $F\simeq ku\vee \Sigma^2 ku$ from Section \ref{sec:type}. We have a map $\ell: \Sigma^\infty BU(1)\rightarrow F$ adjoint to the canonical map $BU(1)\rightarrow BU\simeq \Omega^\infty \Sigma^2 ku$. Similarly we have a map $\ell': \Sigma^\infty BU(1)\rightarrow E$ corresponding to the inclusion $BU(1)\rightarrow B_{\com}U$. Now let $s: F\rightarrow E$ be the splitting defined in Remark \ref{rem:splittingmap}. Under the equivalences $F\simeq ku\wedge S^2_+$ and $E\simeq ku\wedge BU(1)_+$ it is the map induced by the canonical map $S^2\rightarrow BU(1)$. We claim that the restriction of the difference class
\[
\ell'-s\circ \ell\in \tilde{K}_{\com}(BU(1))
\]
along the inclusion $j: \mathbb{C}P^2\hookrightarrow \mathbb{C}P^\infty\simeq BU(1)$ descends to a unique class on $S^4$ which represents the generator $y_2\in \tilde{K}_{\com}(S^4)$.

The cofibre sequence $\mathbb{C}P^1\hookrightarrow \mathbb{C}P^2\stackrel{q}{\longrightarrow} S^4$ yields an exact sequence
\[
0=\tilde{K}_{\com}(\Sigma \mathbb{C}P^1)\longrightarrow \tilde{K}_{\com}(S^4)\stackrel{q^\ast}{\longrightarrow} \tilde{K}_{\com}(\mathbb{C}P^2)\longrightarrow \tilde{K}_{\com}(\mathbb{C}P^1)\,,
\]
which shows that a class in the kernel of the restriction map to $\mathbb{C}P^1$ comes from a unique class in $\tilde{K}_{\com}(S^4)$. Since $\tilde{K}_{\com}(\mathbb{C}P^1)\cong \tilde{K}(\mathbb{C}P^1)$, the image of $\ell'-s\circ \ell$ in $\tilde{K}_{\com}(\mathbb{C}P^1)$ is zero. Thus there exists a unique $y\in \tilde{K}_{\com}(S^4)$ so that $q^\ast(y)=j^\ast(\ell'-s\circ \ell)$. By construction, $y$ lies in the kernel of the projection map $\tilde{K}_{\com}(S^4)\rightarrow \tilde{K}(S^4)$, so its component in the `$uy_1$-direction' is zero. To determine the `$y_2$-component` we note that the cohomology class $\lambda_2$ can be computed from the map of spectra
\[
ku\wedge BU(1)_+\simeq \bigvee_{n\geq 0} \Sigma^{2n}ku\rightarrow  \Sigma^4ku\rightarrow \Sigma^4 H\mathbb{Z}\,,
\]
where $ku\rightarrow H\mathbb{Z}$ is the standard map. That is, we regard $\lambda_2$ as a natural map $\tilde{K}_{\com}(-)\rightarrow \tilde{H}^4(-,\mathbb{Z})$. Since $s: F\rightarrow E$ only hits the wedge summands for $n=0,1$ this shows that $\lambda_2(s\circ \ell)=0$ in $\tilde{H}^4(BU(1),\mathbb{Z})$. Furthermore, we see (\emph{e.g.} from the proof of Lemma \ref{lem:bcomutou2}) that $\lambda_2(\ell')$ is the canonical generator of $\tilde{H}^4(BU(1),\mathbb{Z})$. By naturality, then $\lambda_2(y)$ is a generator of $\tilde{H}^4(S^4,\mathbb{Z})$. This shows that $y$ and $y_2$ agree as elements of $\tilde{K}_{\com}(S^4)$.

\appendix

\section{The map $B_{\com}G_{\mathds{1}}\rightarrow BG$ on rational cohomology} \label{sec:modulestructure}

Recall the definition of $B_{\com}G_{\mathds{1}}$ from Section \ref{sec:cohomologybcomg}. The inclusion map $i: B_{\com}G_{\mathds{1}}\rightarrow BG$ induces a ring homomorphism
\[
i^\ast: H^\ast(BG)\rightarrow H^{\ast}(B_{\com}G_{\mathds{1}})\,,
\]
which gives $H^\ast(B_{\com}G_{\mathds{1}})$ the structure of a $H^\ast(BG)$-module \cite{ademcommutativity}. When rational coefficients are used, the map $i^\ast$ can be completely described by means of the rational isomorphism (\ref{eq:bcomgcohomology}) and a general observation about the conjugation map that we shall now explain.

Let $G$ be a compact Lie group and $H\leqslant G$ an abelian, closed subgroup. Let $BH$ and $BG$ denote the bar construction, and let $j: BH\rightarrow BG$ be the map induced by the inclusion $H\hookrightarrow G$. Let us regard $G/H$ as a constant simplicial space. Recall that a map
\begin{equation} \label{eq:conjugationmap}
\alpha: G/H\times BH\rightarrow BG
\end{equation}
can be defined simplicialwise by letting
\[
(gH,h_1,\dots,h_k)\stackrel{\alpha}{\longmapsto} (gh_1g^{-1},\dots,gh_kg^{-1})
\]
on $k$-simplices.

In the following proposition we give an alternative description of the map (\ref{eq:conjugationmap}) up to homotopy, and determine its homotopy fibre. Let $EH\times_H G$ be the one-sided bar construction for the action of $H$ on $G$ by left-translation. Then we have the standard homotopy fibre sequence
\[
EH\times_{H}G \stackrel{l}{\longrightarrow} BH \stackrel{j}{\longrightarrow} BG\,,
\]
where $l$ is the evident projection map. Note that $EH\times_H G\simeq G/H$. The equivalence is $(h_1,\dots,h_k\,|\,g)\mapsto g^{-1}H$ on $k$-simplices. By abuse of notation, we shall also write $l: G/H\rightarrow BH$ for the map obtained by implicitly inverting this equivalence.

Since $H$ is assumed abelian, there is a multiplication map $\mu: BH\times BH\rightarrow BH$.

\begin{proposition} \label{prop:conjugationmap}
The conjugation map $\alpha: G/H\times BH\rightarrow BG$ is homotopic to the composition of maps
\[
G/H\times BH\xrightarrow{l\times id} BH\times BH\stackrel{\mu}{\longrightarrow} BH\stackrel{j}{\longrightarrow} BG\, .
\]
Moreover, there is a homotopy fibre sequence
\[
G/H\times G/H \stackrel{sh}{\longrightarrow} G/H\times BH\stackrel{\alpha}{\longrightarrow} BG\,,
\]
where the map $sh$ is the `shear' map given by
\[
sh(gH,g'H)=(gH, \mu(l(gH)^{-1},l(g'H)))\,.
\]
\end{proposition}
\begin{proof}
To prove the first part of the proposition it suffices to show that the two principal $G$-bundles over $G/H\times BH$ classified by $j\circ \mu\circ (l\times id)$ respectively $\alpha$ are isomorphic. Thus, we must compare the pullbacks of the universal bundle $\pi: EG\rightarrow BG$ under both maps. It is known that a levelwise pullback of simplicial spaces remains a pullback after geometric realisation. Then it is easily seen that the following diagram is a pullback square by verifying the pullback simplicialwise,
\begin{equation} \label{dgr:homotopyfibreconjugationmap} \xymatrixcolsep{5pc}
\xymatrix{
E(H\times H)\times_{\rho} (G\times G) \ar[r]^-{p} \ar[d]^-{q} & EG \ar[d]^-{\pi}	\\
(EH\times_{H}G)\times BH \ar[r]^-{j\circ \mu\circ (l\times id)}	& BG\,.
}
\end{equation}
In the top left corner, the one-sided bar construction is formed by regarding $G\times G$ as a left $H\times H$-space via the action
\begin{alignat*}{2}
\rho: H\times H\times G\times G	&\quad&	\longrightarrow	&\quad	G\times G			\\
(h,h',g,g')						&\quad&	\longmapsto	&\quad	(hg,hh'g')\, .
\end{alignat*}
Right-multiplication on the second factor of $G$ makes the bar construction the total space of a principal $G$-bundle. The bundle projection $q$ is given by two component maps: The component into $EH\times_H G$ is obtained by projecting onto the first factors. The component map into $BH$ is the projection onto $B(H\times H)$ followed by projection onto the second factor. The bundle map $p$ is induced by mutliplication in $H$ and projection onto the second factor of $G$.

On the other hand, pulling back the universal $G$-bundle along the map $\alpha$ produces the following diagram
\[
\xymatrix{
E(H\times H)\times_{\rho'} (G\times G) \ar[r]^-{p'} \ar[d]^-{q} & EG \ar[d]^-{\pi}	\\
(EH\times_{H}G)\times BH \ar[r]^-{\alpha}	& BG\, .
}
\]
The action $\rho'$ is given by
\begin{alignat*}{2}
\rho': H\times H\times G\times G	&\quad&	\longrightarrow	&\quad	G\times G			\\
(h,h',g,g')						&\quad&	\longmapsto	&\quad	(hg,g^{-1}h'gg')\, ,
\end{alignat*}
and the bundle map $p'$ is most easily described on $k$-simplices, where it is given by
\[
((h_1,h_1'),\dots,(h_k,h_k')\,|\,(g,g'))\stackrel{p'}{\longmapsto} (g^{-1}h_1'g,\dots,g^{-1}h_k'g\,|\,g')\, .
\]
The bundle projection $q$ is the same as in (\ref{dgr:homotopyfibreconjugationmap}).

It remains to compare the resulting principal $G$-bundles over the common base space $(EH\times_{H}G)\times BH$. One can check that the shear map
\begin{alignat*}{2}
G\times G	&\quad&	\longrightarrow	&\quad	G\times G			\\
(g,g')		&\quad&	\longmapsto	&\quad	(g,g^{-1}g')
\end{alignat*}
is $H\times H$-equivariant for the actions $\rho$ and $\rho'$ and induces an isomorphism of the two principal bundles. This proves that $\alpha \simeq j\circ \mu\circ (l\times id)$.

To prove the second part of the proposition we may take the homotopy fibre of $\alpha$ to be the space in the top left corner of (\ref{dgr:homotopyfibreconjugationmap}). We have the following commutative diagram,
\[
\xymatrix{
(EH\times_H G)\times (EH\times_H G) \ar[r]^-{\cong}	_-{\sigma} \ar@/_1pc/[dr]^-{sh}	& \;E(H\times H)\times_{\rho} (G\times G) \ar[d]^-{q}	\\
													& (EH\times_{H}G)\times BH\,,
}
\]
where the horizontal isomorphism is given by
\[
((h_1,\dots,h_k\,|\, g),(h_1',\dots,h_k'\,|\, g'))\stackrel{\sigma}{\longmapsto} ((h_1,h_1^{-1}h_1'),\dots,(h_k,h_k^{-1}h_k')\,|\, (g,g'))
\]
on $k$-simplices. The composite map $sh:=q\circ \sigma$ is then given by
\[
((h_1,\dots,h_k\,|\, g),(h_1',\dots,h_k'\,|\, g'))\stackrel{sh}{\longmapsto} ((h_1,\dots,h_k\,|\, g),(h_1^{-1}h_1',\dots,h_k^{-1}h_k'))\,.
\]
Under the equivalence $EH\times_H G\simeq G/H$ it becomes the shear map $sh$ as described in the proposition.
\end{proof}

Let $H^\ast$ denote cohomology with $\mathbb{Q}$-coefficients.\bigskip

Consider the case where $H=T$ is a maximal torus in $G$ and $W=N(T)/T$ is its Weyl group. The multiplication $\mu: BT\times BT\rightarrow BT$  makes the algebra $H^\ast(BT)$ a Hopf algebra with comultiplication $\mu^\ast$.

\begin{corollary} \label{cor:inclusionmaponcohomology}
Under the well known isomorphism $H^\ast(BG)\cong H^\ast(BT)^W$ and the isomorphism (\ref{eq:bcomgcohomology2}) the map $i^\ast: H^\ast(BG)\rightarrow H^\ast(B_{\com}G_{\mathds{1}})$ is given by
\[
H^\ast(BT)^W\stackrel{\mu^\ast}{\longrightarrow} (H^\ast(BT)\otimes H^\ast(BT))^W\xrightarrow{\textnormal{proj.}} (H^\ast(BT)\otimes H^\ast(BT))^W/J\, .
\] 
\end{corollary}
\begin{proof}
This is immediate from Proposition \ref{prop:conjugationmap}.
\end{proof}

\begin{example} \label{ex:cohomologybcomun}
Consider the case $G=U(k)$. We have recalled the rational cohomology of $B_{\com}U(k)$ in Section \ref{sec:cohomologybcomg}. The cohomology ring is generated by the multisymmetric functions (\ref{eq:generatorzabn}), that is, by the
\[
z_{a,b,k}=x_1^ay_1^b+\dots +x_k^ay_k^b
\]
for $b\geq 1$. Similarly, $H^\ast(BU(k))\cong H^\ast(BT(k))^W=\mathbb{Q}[t_1,\dots,t_k]^{\Sigma_k}$ is generated by the polynomials $z_n:=t_1^n+\dots+t_k^n$. Since each $t_i$ is primitive, we have
\[
\mu^\ast(z_n)=\sum_{i=1}^k(t_i\otimes 1+1\otimes t_i)^n=\sum_{j=0}^{n}\binom{n}{j}\sum_{i=1}^k t_i^{j}\otimes t_i^{n-j}\, .
\]
Identifying $t_i^j\otimes t_i^{n-j}$ with $x_i^jy_i^{n-j}$ in $H^\ast(B_{\com}U(k))$ and applying Corollary \ref{cor:inclusionmaponcohomology} shows that
\begin{equation} \label{eq:ibcomun}
i^\ast(z_n)=\sum_{j=0}^{n-1} \binom{n}{j} z_{j,n-j,k}\, .
\end{equation}
In the limit $k\to \infty$ we arrive at the formula derived in Example \ref{ex:inclusionbcomucoh}.
\end{example}

\bibliographystyle{amsplain} 

\providecommand{\bysame}{\leavevmode\hbox to3em{\hrulefill}\thinspace}
\providecommand{\MR}{\relax\ifhmode\unskip\space\fi MR }
\providecommand{\MRhref}[2]{%
  \href{http://www.ams.org/mathscinet-getitem?mr=#1}{#2}
}
\providecommand{\href}[2]{#2}

\vskip .3in
\noindent
Mathematical Institute
\newline
University of Oxford
\newline
Andrew Wiles Building
\newline
Oxford OX2 6GG
\newline
UK \bigskip

\noindent
{\it gritschacher@maths.ox.ac.uk}

\end{document}